\documentclass[11 pt]{amsart}     
\usepackage{amsthm}
\usepackage{amssymb}
\usepackage{latexsym}
\usepackage{amsfonts}
\usepackage{amsmath}
\usepackage{amstext}
\usepackage{graphicx,epsfig}

\usepackage[shortalphabetic,abbrev]{amsrefs}
\usepackage{hyperref}
\newtheorem{theorem}{Theorem}
\newtheorem{lemma}[theorem]{Lemma}
\newtheorem{corollary}[theorem]{Corollary}
\newtheorem{proposition}[theorem]{Proposition}
\theoremstyle{definition}
\newtheorem{definition}[theorem]{Definition}
\theoremstyle{remark}
\newtheorem*{remark}{Remark}
\newcommand{\RR}{\mathbb{R}}

\newcommand{\ZZ}{\mathbb{Z}}

\newcommand{\pd}[2][1]{\deriv{\partial}{#1}{#2}}
\newcommand{\deriv}[3]{\ensuremath{\frac{#1 #2}{#1 #3}}}
\newcommand{\abs}[1]{\ensuremath{\left|#1\right|}}

\newcommand{\E}{\mathrm{e}}
\newcommand{\CSF}{\textsc{csf}}
\newcommand{\curv}{\kappa}
\newcommand{\nor}{\textsf{N}}
\newcommand{\tang}{\textsf{T}}
\newcommand{\rr}{\frak r}
\newcommand{\eps}{\varepsilon}

\begin{document}

\title[Isoperimetric profile comparison for curve shortening flow]{A comparison theorem for the isoperimetric profile under curve shortening flow}
\thanks{Partially supported by Discovery Grant DP0985802 of the Australian Research Council.}

\author{Ben Andrews}

\address{Mathematical Sciences Institute,  Australian National University;
		Mathematical Sciences Center, Tsinghua University; and
		Morningside Center of Mathematics, Chinese Academy of Sciences.}
              
              \email{Ben.Andrews@anu.edu.au}           
\author{Paul Bryan}
\address{Mathematical Sciences Institute, Australian National University}
\email{Paul.Bryan@anu.edu.au}

\date{}

\begin{abstract}
We prove a comparison theorem for the isoperimetric profiles of simple closed curves evolving by the normalized curve shortening flow:  If the isoperimetric profile of the region enclosed by the initial curve is greater than that of some `model' convex region with exactly four vertices and with reflection symmetry in both axes, then the inequality remains true for the isoperimetric profiles of the evolved regions.    We apply this using the Angenent solution as the model region to deduce sharp time-dependent upper bounds on curvature for arbitrary embedded closed curves evolving by the normalized curve shortening flow.  A slightly different comparison also gives lower bounds on curvature, and the result is a simple and direct proof of Grayson's theorem without use of any blowup or compactness arguments, Harnack estimates, or classification of self-similar solutions.
\end{abstract}

\maketitle

\section{Introduction}
The curve shortening flow (\CSF) produces a smooth family of curves $\tilde\gamma_t=\tilde X(S^1,t)$ in the plane $\RR^2$,  from an initial curve $\tilde\gamma_0$ given by an immersion $\tilde X_0:\ S^1\to\RR^2$, according to the equation
\begin{equation}\label{eq:CSF}
\frac{\partial \tilde X}{\partial \tau} = -\tilde{\curv}\nor=\frac{1}{\abs{\tilde X'}}\left(\frac{\tilde X'}{\abs{\tilde X'}}\right)'
\end{equation}
where $\tilde{\curv}$ is the curvature of the curve $\tilde\gamma_\tau$, $\nor$ is the outward unit normal, and primes denote derivatives with respect to a local parameter on $S^1$.  This system has received considerable study, and in particular it is known that for any smooth immersion $X_0$ there exists a unique solution on a finite maximal time interval, and that the maximum curvature becomes unbounded as the maximal time is approached \cite{GH}.  Gage \cites{Ga1,Ga2} and Gage and Hamilton \cite{GH} considered the case of convex embedded closed curves, and proved that solutions are asymptotic to shrinking circles as the final time is approached.  Grayson \cite{Gr} then extended this result to arbitrary embedded closed curves.  Our aim in this paper is to provide an estimate on the curvature for embedded closed curves evolving by curve shortening flow, and deduce from this a simple proof of Grayson's theorem.

In recent work \cite{AB1} we used isoperimetric estimates to deduce curvature bounds for embedded solutions of curve shortening flow.  The result was obtained by controlling the lengths of chords to the evolving curves, as a function of the arc length between the end points and elapsed time, extending an idea introduced by Huisken \cite{HuCSF}.  In particular sufficiently strong control of chord length for short segments implies a curvature bound, strong enough to provide a rather simple proof of Grayson's theorem.  Our argument showed that chord lengths can be bounded from below by a function $f(\ell,t)$ of arc length $\ell$ and elapsed time $t$, provided $f$ satisfies a certain differential inequality.  We then produced an explicit solution of this inequality which we discovered purely by accident, and for which we have no simple motivation.

Subsequently \cite{AB2} we used similar ideas to give sharp curvature estimates for the normalized Ricci flow on the two-sphere.  As before, the key motivating idea is that sufficiently strong control on an isoperimetric profile implies control on curvatures, but in this case we no longer relied on a purely serendipitous calculation:  The solutions of the differential equality in that case are in direct correspondence to axially symmetric solutions of the normalized Ricci flow itself, and in particular explicit solutions could be constructed from an explicit solution of Ricci flow known as the Rosenau solution or `sausage model'.

In this paper we show that the same situation arises in curve shortening flow when one estimates the isoperimetric profile of the enclosed region (related estimates were used by Hamilton to rule out slowly forming singularities \cite{HamCSF}).  As in the Ricci flow case, we deduce a comparison result from any solution of a certain differential inequality, and solutions of the corresponding equality are in direct correspondence with symmetric solutions of the curve shortening flow.  Using the explicit symmetric solution constructed by Angenent (known as the `paperclip' solution, we deduce a very strong upper bound on curvature for an arbitrary embedded solution of curve shortening flow.

A new ingredient which arises here is that the isoperimetric estimate does not imply lower bounds on the curvature $\curv$ (in contrast to the result in \cite{AB1} where a bound on $\curv^2-1$ is deduced for normalized solutions).  However we can deduce a suitable lower bound on $\curv$ by estimating the isoperimetric profile of the exterior region, and indeed the lower bounds we obtain (produced by comparison with a self-similar expanding solution) have some similarity to those which arise in Ricci flow.

\section{Notation and preliminary results}
To set our conventions, we routinely parametrize simple closed curves in the anticlockwise direction with outward-pointing normal, which means the Serret-Frenet equations take the form
\begin{align*}
X' &=|X'|\tang;\\
\tang'&=-\curv|X'|\nor;\\
\nor' &= \curv |X'|\tang.
\end{align*}

Our result is most easily formulated in terms of a normalized version of the curve-shortening flow, which we now introduce:  Given a solution $\tilde X$ of \eqref{eq:CSF}, we define $X:\ S^1\times[0,T)\to\RR^2$ by
$$
X(p,t)=\sqrt{\frac{\pi}{A[\tilde\gamma_\tau]}}\tilde X(p,\tau),
$$
where $A[\tilde\gamma_\tau]$ is the area enclosed by the curve $\tilde\gamma_\tau$, and
$$
t = \int_0^\tau \frac{\pi}{A[\tilde\gamma_{\tau'}]}\,d\tau',\qquad\text{and}\qquad
T=\int_0^{\tilde T}\frac{\pi}{A[\tilde\gamma_{\tau'}]}\,d\tau'.
$$
Then the rescaled curve $\gamma_t=X(S^1,t)$ has $A[\gamma_t]=\pi$ for every $t$, and $X$ evolves according to the normalized equation
\begin{equation}
  \label{eq:NCSF}
  \pd[X]{t} = X - \curv \nor = X+\frac{1}{|X'|}\left(\frac{X'}{|X'|}\right)'
\end{equation}
where $\curv$ denotes the curvature of $\gamma_t$.  Our main result controls the behaviour of solutions of \eqref{eq:NCSF} via their isoperimetric profiles, which we now discuss.

Let $\Omega$ be an open subset of $\RR^2$ of area $A$ (possibly infinite) with smooth boundary curve $\gamma$.   The \emph{isoperimetric profile} of $\Omega$ is the function $\Psi:\ (0,A)\to\RR_+$ defined by
\begin{equation}\label{eq:IP}
\Psi(\Omega,a) = \inf\left\{\abs{\partial_{\Omega}K}:\ K\subseteq\Omega,\ \abs{K}=a\right\}.
\end{equation}
Here $\partial_\Omega K$ denotes the boundary of $K$ as a subset of $\Omega$, which is given by the part of the boundary of $K$ as a subset of $\RR^2$ which is not contained in $\gamma$.  If $\partial\Omega$ is compact, then
for each $a\in(0,A)$, equality in the infimum is attained for some $K\subseteq\Omega$, so that we have $\abs{K}=a$ and $\abs{\partial_{\Omega}K}=\Psi(a)$, and in this case $\partial_{\Omega}K$ consists of circular arcs of some fixed radius meeting $\gamma$ orthogonally.

Later in the paper we will also consider the \emph{exterior isoperimetric profile} $\Psi_\text{ext}(\Omega,.)$, which is simply the isoperimetric profile of the exterior of $\Omega$:  $\Psi_\text{ext}(\Omega,a)=\Psi(\RR^2\setminus\bar\Omega,a)$.  If $\Omega$ is compact with smooth boundary, the exterior isoperimetric profile is defined on $[0,\infty)$, and for each $a>0$ there is some region $K$ in the exterior of $\Omega$ which attains the isoperimetric profile in the sense that $|K|=a$ and $|\partial_{\RR^2\setminus\bar\Omega}K|=\Psi_{\text{ext}}(\Omega,a)$.

\begin{proposition}\label{prop:asympt.isoperim.profile}
For any smoothly bounded domain $\Omega$ of area $\pi$, we have
$$
\lim_{a\to 0}\frac{\Psi(\Omega,a)-\sqrt{2\pi a}}{a}=-\frac{4\sup_{\partial\Omega}\curv}{3\pi};
\quad
\lim_{a\to 0}\frac{\Psi_{\text{ext}}(\Omega,a)-\sqrt{2\pi a}}{a}=\frac{4\inf_{\partial\Omega}\curv}{3\pi}.
$$
\end{proposition}

\begin{proof}
In the case $\Omega=B_1(0)$ we can check this result explicitly, since the isoperimetric regions are precisely the disks and half-spaces which intersect $B_1(0)$ orthogonally, so that the isoperimetric profile is given implicitly by 
$$
a=\theta-\tan\theta+(\pi/2-\theta)\tan^2\theta\quad\text{and}\quad
\Psi(B_1(0),a)=(\pi-2\theta)\tan\theta,
$$
from which the asymptotic result $\Psi(B_1(0),a)=\sqrt{2\pi a}-\frac{4a}{3\pi}+O(a^{3/2})$ follows.  

The exterior isoperimetric profile can be computed similarly:  In this case the isoperimetric regions are the intersections with $\RR^2\setminus B_1(0)$ of disks which meet the boundary orthogonally, so the exterior isoperimetric profile is defined implicitly by the identities
$$
a=\tan\theta-\theta+(\pi/2+\theta)\tan^2\theta\quad\text{and}\quad
\Psi_{\text{ext}}(B_1(0),a)=(\pi-2\theta)\tan\theta.
$$

By scaling, we have also that the isoperimetric profiles for a ball of radius $r$ are given by
\begin{align*}
\Psi(B_r(0),a) &= r\Psi(B_1(0),a/r^2) = \sqrt{2\pi a}-\frac{4a}{3\pi r}+O(a^{3/2});\\
\Psi_{\text{ext}}(B_r(0),a) &= r\Psi_{\text{ext}}(B_1(0),a/r^2) = \sqrt{2\pi a}+\frac{4a}{3\pi r}+O(a^{3/2}).
\end{align*}
We also note the isoperimetric profile of a half-space: $\Psi(\{x>0\},a)=\sqrt{2\pi a}$.

\begin{figure}
\includegraphics[scale=0.60]{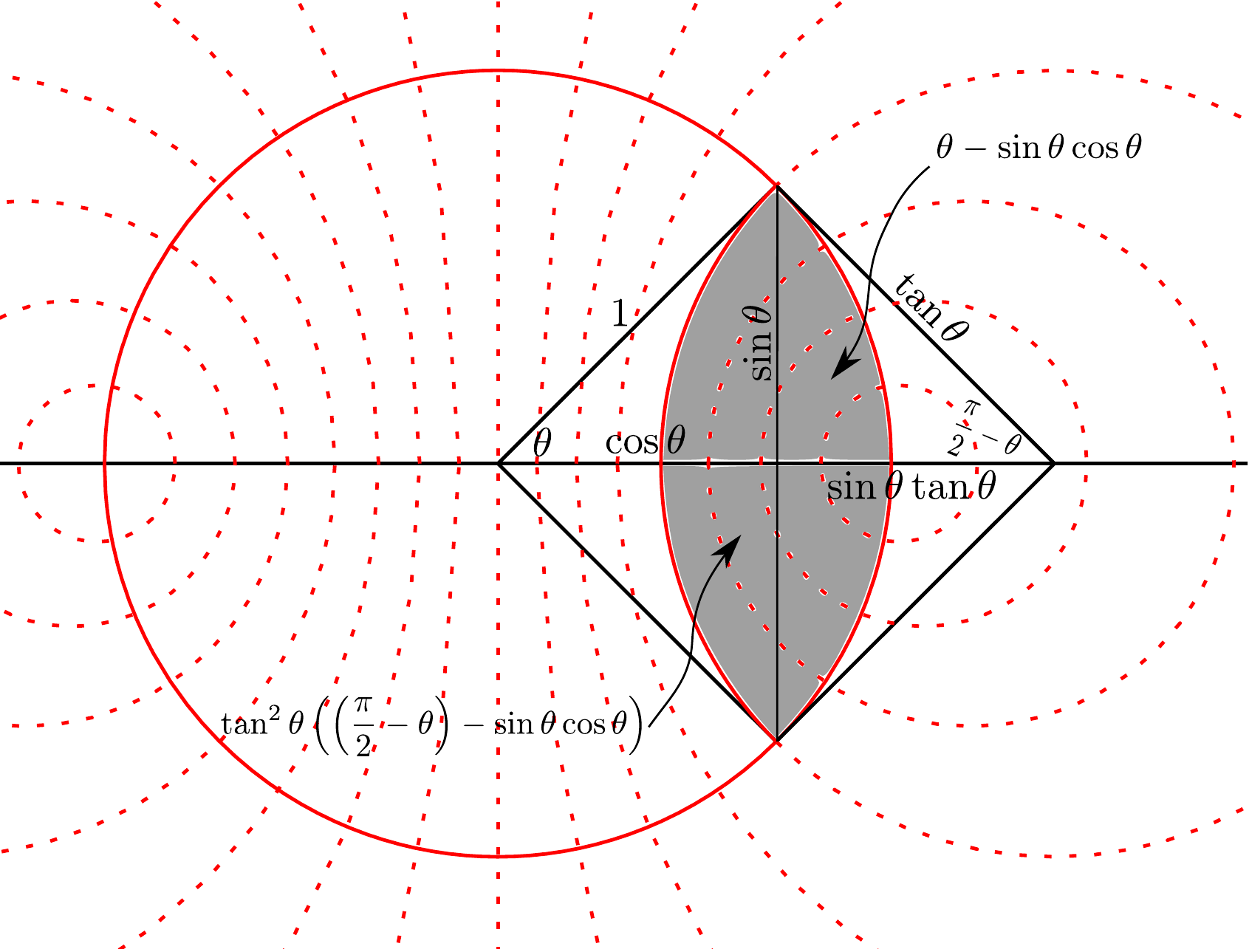}
\caption{Isoperimetric regions of the unit disk}\label{fig:circle_isoperim}
\end{figure}

In the general case,
we begin by proving $\Psi(\Omega,a)\leq\sqrt{2\pi a}+O(a)$:
Let $p\in\partial\Omega$, and set $K_r=B_r(p)\cap\Omega$.
A direct computation gives
$$
|\partial_{\Omega}K_r| = \pi r+O(r^2)
$$
while
$$
|K_r|=\frac{\pi}{2}r^2+O(r^3)
$$
as $r\to 0$.  Setting $a=|K_r|$ and rearranging, we find
\begin{equation}\label{eq:first.isoper.est}
\Psi(\Omega,a)\leq |\partial_{\Omega}K_r| = \sqrt{2\pi a}+O(a)
\end{equation}
as $a\to 0$.  

Now we prove the stronger result:   Let $X:\ \RR\to\partial\Omega$ be a unit speed counterclockwise parametrization of the boundary, and define $Y:(\RR/L\ZZ)\times[0,\delta)\to\Omega$ by $Y(u,s) = X(u)-s\nor(u)$.
For small $\delta>0$ this map parametrizes a neighbourhood of the boundary, with induced metric given by 
\begin{equation}\label{eq:metric}
g(\partial_s,\partial_s)=1,\quad g(\partial_s,\partial_u)=0,\quad g(\partial_u,\partial_u) = (1-s\curv(u))^2.
\end{equation}
For $\curv\in\RR$ we define a `model' region $\Omega_\curv$ with the origin in its boundary:
$$
\Omega_\curv=\begin{cases}
\{(x,y):\ x\leq 0\},&\curv=0;\\
B_{\curv^{-1}}(-\curv^{-1},0),&\curv>0;\\
\RR^2\setminus(B_{|\curv|^{-1}}(|\curv|^{-1},0)),&\curv<0.
\end{cases}
$$
For any $\bar u\in\RR$, we can construct a local diffeomorphism $\chi$ from a neighbourhood of $X(u_0)$ in $\Omega$ to a neighbourhood of the origin in $\Omega_{\curv}$, as follows:
$$
\chi(Y(u,s)) = \begin{cases}
-s+(u-\bar u)i,&\curv(\bar u)=0;\\
(\curv(\bar u)^{-1}-s)\E^{i\curv(\bar u)(u-\bar u)}-\curv(\bar u)^{-1},&\curv(\bar u)\neq 0;\\
\end{cases}
$$
We see from \eqref{eq:metric} that $\chi$ is nearly an isometry, in the sense that there exists $r>0$ such that $\chi$ maps $B_r(X(u_0))\cap\Omega$ to a neighbourhood $U$ of the origin in $\Omega_\curv$ in such a way that $g(1-Cd^2)\leq \chi_*g\leq g(1+Cd^2)$, where $d$ is the distance to $X(u_0)$ (comparable to $s+|u-\bar u|$) and $g$ is the standard metric on $\RR^2$.
We prove an upper bound on the isoperimetric profile as follows:  For $a$ sufficiently small, we can find an isoperimetric domain $K$ for $\Omega_{\curv}$ contained in $U$ such that $\chi^{-1}(K)$ has area $a$ (hence $K$ has area at least $a(1-Ca)$).  But then we have
\begin{align*}
\Psi(\Omega,a)
&\leq |\partial_\Omega\chi^{-1}(K)|_g\\
&=|\partial_{\Omega_{\curv}}K|_{\chi_*g}\\
&=\Psi(\Omega_{\curv},|K|)\\
&\leq\sqrt{2\pi|K|}-\frac{4\curv|K|}{3\pi}+C|K|^{3/2}\\
&\leq\sqrt{2\pi a}-\frac{4\curv a}{3\pi}+\tilde Ca^{3/2}
\end{align*}
The reverse inequality is proved similarly:  By the estimate \eqref{eq:first.isoper.est}, for $a$ small the isoperimetric domain $K$ for $\Omega$ of area $a$ is contained in the domain of the map $\chi$ centred at some point $X(u_0)$.
Then we have 
\begin{align*}
\Psi(\Omega,a)
&= |\partial_\Omega K|_g\\
&=|\partial_{\Omega_\curv}\chi(K)|_{\chi^{-1}_*g}\\
&\geq |\partial_{\Omega_\curv}\chi(K)|_g(1-Ca)\\
&\geq (1-Ca)\Psi(\Omega_\curv,|\chi(K)|)\\
&\geq (1-Ca)\left(\sqrt{2\pi|\chi(K)|}-\frac{4\curv|\chi(K)|}{3\pi}-C|\chi(K)|^{3/2}\right)\\
&\geq \sqrt{2\pi a}-\frac{4\curv a}{3\pi}-\tilde Ca^{3/2},
\end{align*}
where we used $|\chi(K)|_g=|K|_{\chi_*g}\geq |K|_g(1-Ca)=a(1-Ca)$.
\end{proof}



\section{A comparison theorem for the isoperimetric profile}

In this section we show that the isoperimetric profile of a region evolving by \eqref{eq:NCSF} can be bounded below by any function satisfying a certain differential inequality, provided this is true at the initial time.  In the following section we will show how to construct such functions from particular solutions of the normalized curve shortening flow.  In order to state the main result of this section we first require the following definition:

\begin{definition}\label{def:F}
For $a,b\in\RR$, we define
\begin{align*}
{\mathcal F}[a,b] &= \inf\left\{\int_0^1\left|\frac{\partial\varphi}{\partial x}\right|^2\,dx - a^2\int_0^1\varphi^2\,dx-b\left(\int_0^1\varphi\,dx\right)^2:\right.\\
&\qquad\qquad\qquad\null\left. \varphi\in C^\infty([0,1]),\ \varphi(0)=\varphi(1)=1\phantom{\int_0^1 \eta\,du}\qquad\right\}.
\end{align*}
\end{definition}

A direct computation shows that
\begin{equation}\label{eq:expr.F}
\frac{1}{{\mathcal F}[a,b]}=\min\left\{\frac{\cos(a/2)}{2a\sin(a/2)}-\frac{1}{a^2}+\frac{1}{a^2+b},0\right\},
\end{equation}
where this should be interepreted as a suitable limit in the case $a=0$.
In particular, in the region where ${\mathcal F}[a,b]$ is positive, it is a smooth function of $a$ and $b$ which is strictly decreasing in $b$.

\begin{theorem}\label{thm:ODEcomparison}
Let $f:\ [0,\pi]\times[0,\infty)\to\RR$ be continuous, smooth on $(0,\pi)\times(0,\infty)$, concave in the first argument for each $t$, and symmetric (so that $f(z,t)=f(\pi-z,t)$ for all $z,t$).  Assume that $\limsup_{z\to 0}\frac{f(z,t)}{\sqrt{2\pi z}}<1$ and
$$
\frac{\partial f}{\partial t}<-f^{-1}{\mathcal F}[ff',f^3f'']+f+f'(\pi-2a)-f(f')^2
$$
for all $a\in(0,\pi)$ and $t> 0$.
Suppose $\gamma_t=\partial\Omega_t$ is a family of smooth embedded curves evolving by \eqref{eq:NCSF}  and satisfying $\Psi(\Omega_0,a)>f(a,0)$ for all $a\in(0,\pi)$, then $\Psi(\Omega_t,a)>f(a,t)$ for all $t\geq 0$ and $a\in(0,\pi)$.
\end{theorem}

\begin{proof}
We argue by contradiction:  If the inequality $\Psi(\Omega_t,a)>f(a,t)$ does not hold everywhere, then define $t_0=\inf\{t:\ \Psi(\Omega_t,a)\leq f(a,t)\text{\ for\ some\ }a\in(0,\pi)\}$.   Since  $\Psi(\Omega_t,a)$ is continuous in $a$ and $t$, and $\Psi(\Omega_t,a)>f(a,t)$ for $a$ sufficiently close to either $0$ or $\pi$, we have $\Psi(\Omega_{t},a)\geq f(a,t)$ for all $a\in[0,\pi]$ and $0\leq t\leq t_0$, and there exists $a_0\in(0,\pi)$ such that $\Psi(\Omega_{t_0},a_0)=f(a_0,t_0)$.  Let $K$ be an isoperimetric region in $\Omega_{t_0}$ of area $a_0$, so that $|\partial_{\Omega_{t_0}}K|=f(|K|,t_0)$.  

The concavity of $f$ has topological implications for $K$:

\begin{lemma}\label{lem:connecteddomains}
Let $f:\ (0,\pi)\to\RR$ be positive, strictly concave, and symmetric in the sense that $f(\pi-x)=f(x)$ for each $x$.  If $\Omega\subset\RR^2$ is a compact simply connected domain of area $\pi$ with $\Psi(\Omega, a)\geq f(a)$ for every $a$, then every region $K$ in $\Omega$ with $|\partial_\Omega K|=f(|K|)$ and $|K|\in(0,\pi)$ is connected and simply connected.
\end{lemma}

\begin{proof}
We first prove that $K$ is connected, by contradiction:  Suppose $K_1$ and $K_2$ are nonempty open subsets of $K$ with $K=K_1\cup K_2$, then we have
\begin{align*}
f(|K|)&=|\partial_\Omega K|\\
&=|\partial_\Omega K_1|+|\partial_\Omega K_2|\\
&\geq f(|K_1|)+f(|K_2|)\\
&> f(0)+f(|K_1|+|K_2|)\\
&\geq f(|K|),
\end{align*}
where the strict inequality follows from the strict concavity of $f$.
This is a contradiction, so $K$ is connected.

Since $|\partial_\Omega\left(\Omega\setminus \bar K\right)|=|\partial_\Omega K|=f(|K|)=f(\pi-|K|)=f(|\Omega\setminus \bar K|)$, the same argument implies that $\Omega\setminus\bar K$ is  connected. It follows that $\partial_\Omega K$ has only one component and that $K$ is simply connected.
\end{proof}

\begin{lemma}[First variation]\label{lem:first-var}
$\partial_{\Omega_{t_0}}K$ has constant curvature equal to $f'$.
\end{lemma}

\begin{proof}

\begin{figure}
\includegraphics[scale=0.8]{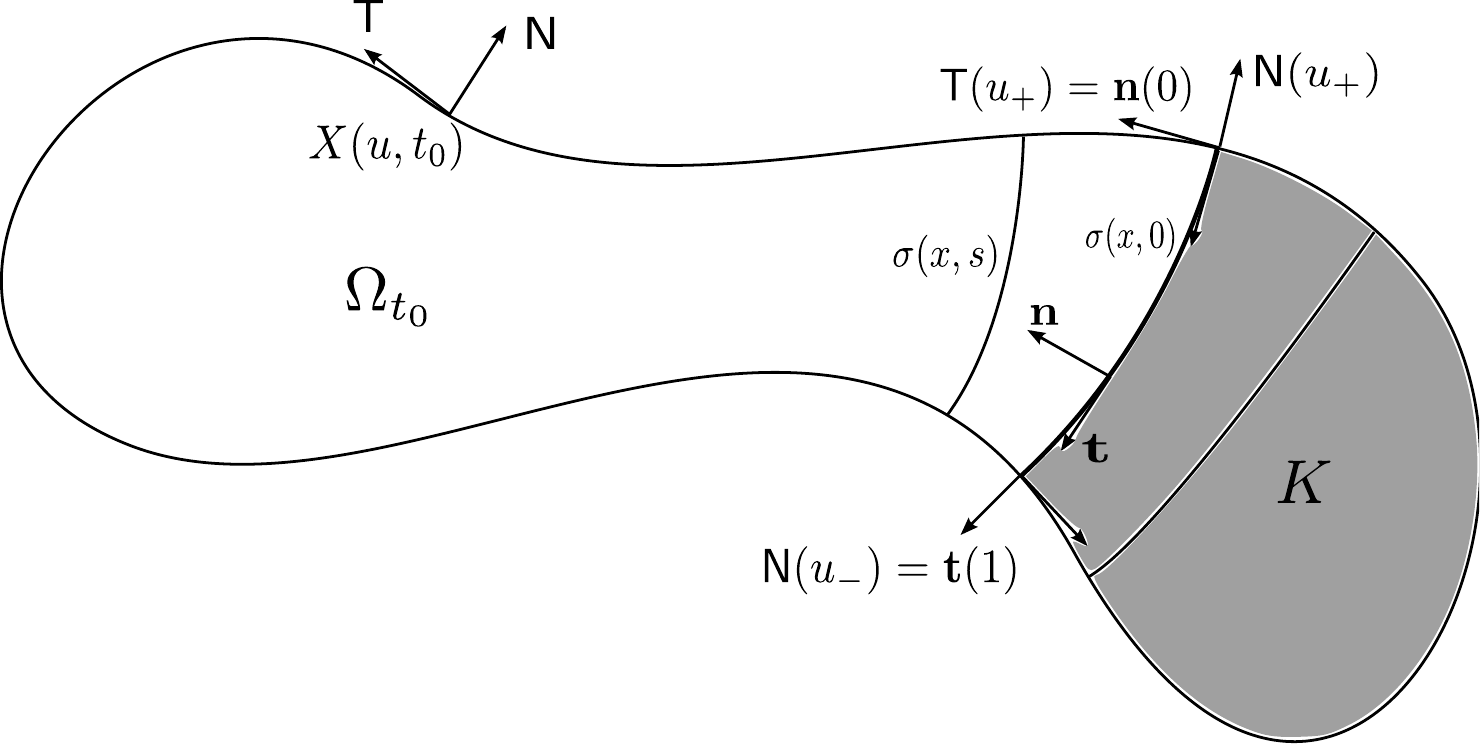}
\caption{A smooth variation of the domain $K$ in $\Omega_{t_0}$.}\label{fig:spatial.variation}
\end{figure}

Given any smooth function $\varphi:\ [0,1]\to\RR$, there exists a smooth variation $\sigma:\ [0,1]\times(-\delta,\delta)\to\Omega_{t_0}$ with $\sigma([0,1],0)=\partial_{\Omega_{t_0}}K$, $\sigma(0,s)=X(u_+(s),t_0)$ and $\sigma(1,s)=X(u_-(s),t_0)$, and such that $\frac{\partial\sigma}{\partial s}(x,0) = \varphi(x){\bf n}$, where ${\bf n}$ is the outward-pointing unit normal to $K$.  Write $\frac{\partial\sigma}{\partial s}=\eta {\bf n}+\xi{\bf t}$, where ${\bf t}=\sigma_x/|\sigma_x|$ is the unit tangent vector, and by assumption $\eta(x,0)=\varphi(x)$ and $\xi(x,0)=0$ for each $x\in[0,1]$.

Let $K_s$ be the region in $\Omega_{t_0}$ bounded by the curve $\sigma(.,s)$ for each $s\in(-\delta,\delta)$.  The area of $K_s$ is given by the following expression:
$$
|K_s|=\frac12\int_0^1 \sigma\times\frac{\partial\sigma}{\partial x}\,dx + \frac12\int_{u_-(s)}^{u_+(s)}X\times\frac{\partial X}{\partial u}\,du
$$
where $\sigma$ is evaluated at $(x,s)$ and $X$ at $(u,t_0)$.  We can assume that the parameter $u$ is chosen to be the arc-length parameter at time $t_0$, so that $\frac{\partial X}{\partial u}=\tang$ everywhere.  Differentiating with respect to $s$, we find:
\begin{align}
\frac{\partial}{\partial s}|K_s| &=\frac12 \int_0^1(\eta{\bf n}+\xi{\bf t})\times{\bf t}|\sigma_x|\,dx
+\frac12\int_0^1\sigma\times\frac{\partial}{\partial x}\left(\eta{\bf n}+\xi{\bf t}\right)\,dx\notag\\
&\quad\null + \frac12\dot u_+ X(u_+)\times\tang(u_+)-\frac12\dot u_- X(u_-)\times\tang(u_-)
\notag\\
&=\int_0^1\eta |\sigma_x|\,dx + \frac12\sigma\times(\eta{\bf n}+\xi{\bf t})\big|_{x=1}- \frac12\sigma\times(\eta{\bf n}+\xi{\bf t})\big|_{x=0}\notag\\
&\quad\null + \frac12\dot u_+ X(u_+)\times\tang(u_+)-\frac12\dot u_- X(u_-)\times\tang(u_-)
\notag\\
&=\int_0^1 \eta(x,s) |\sigma_x|\,dx.\label{eq:first.var.A}
\end{align}
Here dots denote derivatives with respect to $s$.  We integrated by parts and used the identity ${\bf n}\times{\bf t}=1$ to produce the second equality, and the last equality uses the following identities which are proved by differentiating the equations $\sigma(1)=X(u_-)$ and $\sigma(0)=X(u_+)$ with respect to $s$:
\begin{align*}
\left(\eta{\bf n}+\xi{\bf t}\right)\big|_{x=1} &= \frac{\partial\sigma}{\partial s}(1) = \frac{\partial}{\partial s}X(u_-(s)) = 
\dot u_-\tang(u_-);\\
\left(\eta{\bf n}+\xi{\bf t}\right)\big|_{x=0} &= \frac{\partial\sigma}{\partial s}(0) = \frac{\partial}{\partial s}X(u_+(s)) = 
\dot u_+\tang(u_+).
\end{align*}
Next we compute the rate of change of the length of $\sigma([0,1],s)=\partial_{\Omega_{t_0}}K_s$:
\begin{align}
\frac{\partial}{\partial s}|\partial_{\Omega_{t_0}}K_s|
&=\frac{\partial}{\partial s}\int_0^1|\sigma_x|\,dx\notag\\
&=\int_0^1{\bf t}\cdot \partial_x\left(\eta{\bf n}+\xi{\bf t}\right)\,dx\label{eq:firstvar.L.1}\\
&=\int_0^1\eta\curv_\sigma|\sigma_x|+\xi_x\,dx\notag\\
&=\int_0^1\eta\curv_\sigma|\sigma_x|\,dx + \xi\big|_0^1,\label{eq:firstvar.L.2}
\end{align}
where $\curv_\sigma$ is the curvature of $\sigma$.
At $s=0$ we have $\eta=\varphi$ and $\xi=0$, so
$$
\frac{\partial}{\partial s}\left(|\partial_{\Omega_{t_0}}K_s|-f(|K_s|,t_0)\right)\big|_{s=0}
= \int_0^1\varphi\left(\curv_\sigma-f'\right)|\sigma_x|\,dx
$$
Now we observe that $|\partial_{\Omega_{t_0}}K_s|\geq \Psi(\Omega_{t_0},|K_s|)\geq f(|K_s|,t_0)$ for each $s$, with equality for $s=0$.  Therefore the derivative with respect to $s$ vanishes when $s=0$ for any choice of $\varphi$, and it follows that $\curv_\sigma=f'$ at each point of $\sigma_0$.
\end{proof}

\begin{lemma}[Second variation inequality]\label{lem:second-var-ineq}
For any $\varphi:\ [0,1]\to\RR$, 
$$
\curv(u_-)\varphi(1)^2+\curv(u_+)\varphi(0)^2\leq \frac{1}{f}\int_0^1\varphi_x^2\,dx-f(f')^2\int_0^1
\varphi^2\,dx -f^2f''\left(\int_0^1\varphi\,dx\right)^2.
$$
In particular
\begin{equation}\label{eq:sec.var.F}
\curv(u_-)+\curv(u_+)\leq \frac{1}{f}{\mathcal F}(ff',f^3f'').
\end{equation}
\end{lemma}

\begin{proof}
We consider the variations from the proof of the previous lemma.  Differentiating equation \eqref{eq:first.var.A} we find
\begin{align*}
\frac{\partial^2}{\partial s^2}|K_s|\Big|_{s=0} &=\int_0^1\left(\dot\eta+\eta^2\curv_\sigma\right)|\sigma_x|+\eta\xi_x\,dx\\
&=\int_0^1\left(\dot\eta+\eta^2\curv_\sigma\right)|\sigma_x|\,dx.
\end{align*}
To compute the second derivative of the length $|\partial_{\Omega_{t_0}}K_s|$ it is convenient to differentiate equation \eqref{eq:firstvar.L.1}:
\begin{align*}
\frac{\partial^2}{\partial s^2}\left|\partial_{\Omega_{t_0}}K_s\right|&=\int_0^1\frac{\left|\partial_x\left(\eta{\bf n}+\xi{\bf t}\right)\right|^2}{|\sigma_x|}-\frac{\left|{\bf t}\cdot\partial_x\left(\eta{\bf n}+\xi{\bf t}\right)\right|^2}{|\sigma_x|}\,dx\\
&\quad\null+\int_0^1{\bf t}\cdot\partial_x\partial_s\left(\eta{\bf n}+\xi{\bf t}\right)\,dx\\
&=\int_0^1\frac{\left|{\bf n}\cdot\partial_x\left(\eta{\bf n}+\xi{\bf t}\right)\right|^2}{|\sigma_x|}\,dx+{\bf t}\cdot\partial_s\left(\eta{\bf n}+\xi{\bf t}\right)\Big|_0^1\\
&\quad\null + \int_0^1\curv_\sigma{\bf n}\cdot\partial_s\left(\eta{\bf n}+\xi{\bf t}\right)|\sigma_x|\,dx.
\end{align*}
To expand this further we need to compute $\frac{\partial{\bf t}}{\partial s}$:
\begin{align*}
\frac{\partial{\bf t}}{\partial s}&=\frac{\partial}{\partial s}\left(\frac{\sigma_x}{|\sigma_x|}\right)\\
&=\frac{\partial_x(\eta{\bf n}+\xi{\bf t})}{|\sigma_x|}-\frac{{\bf t}\cdot\partial_x(\eta{\bf n}+\xi{\bf t})}{|\sigma_x|}{\bf t}\\
&=\left(\frac{\eta_x}{|\sigma_x|}-k\xi\right){\bf n}.
\end{align*}
It follows that 
$$
\frac{\partial}{\partial s}{\bf n} = -\left(\frac{\eta_x}{|\sigma_x|}-k\xi\right){\bf t},
$$
and hence we have (since $\xi=0$ for $s=0$)
\begin{equation*}
\frac{\partial}{\partial s}\left(\eta{\bf n}+\xi{\bf t}\right)\big|_{s=0}
=\dot\eta{\bf n}+\left(\dot\xi-\frac{\eta\eta_x}{|\sigma_x|}\right){\bf t}.
\end{equation*}
Substituting this above, and using the result of Lemma \ref{lem:first-var}, we deduce:
\begin{equation*}
\frac{\partial^2}{\partial s^2}\left|\partial_{\Omega_{t_0}}K_s\right|\Big|_{s=0}
=\int_0^1\frac{(\partial_x\varphi)^2}{|\sigma_x|}\,dx + f'\int_0^1\dot\eta|\sigma_x|\,dx + \left(\dot\xi - \frac{\eta\eta_x}{|\sigma_x|}\right)\Big|_0^1.
\end{equation*}
Now we observe that differentiating the identity $X(u_+(s))=\sigma(0,s)$ twice with respect to $s$ yields
\begin{align*}
\left(\dot\eta{\bf n}+\left(\dot\xi-\frac{\eta\eta_x}{|\sigma_x|}\right){\bf t}\right)\Big|_{x=0}
&=\partial_s\left(\eta{\bf n}+\xi{\bf t}\right)\Big|_{x=0}\\ 
&= \frac{\partial^2}{\partial s^2}X(u_+)\\
&= \frac{\partial}{\partial s}\left(\dot u_+\tang(u_+)\right)\\
&= \ddot u_+\tang(u_+)-\left(\dot u_+\right)^2\curv(u_+)\nor(u_+).
\end{align*}
At $s=0$ we have $\nor(u_+)=-{\bf t}(0)$ and $\eta{\bf n}|_{x=0}=\dot u_+\tang(a_+)$, so
$$
\left(\dot\xi-\frac{\eta\eta_x}{|\sigma_x|}\right)\Big|_{x=0}=\varphi(0)^2\curv(u_+).
$$
Similarly, we have (since $\nor(u_-)={\bf t}(1)$ and $\eta{\bf n}|_{x=1}=\dot u_-\tang(u_-)$)
$$
\left(\dot\xi-\frac{\eta\eta_x}{|\sigma_x|}\right)\Big|_{x=1}=-\varphi(1)^2\curv(u_-).
$$
Thus the second variation for length becomes
\begin{equation*}
\frac{\partial^2}{\partial s^2}\left|\partial_{\Omega_{t_0}}K_s\right|\Big|_{s=0}
=\int_0^1\frac{(\partial_x\varphi)^2}{|\sigma_x|}\,dx + f'\int_0^1\dot\eta|\sigma_x|\,dx -\varphi(0)^2\curv(u_+)-\varphi(1)^2\curv(u_-).
\end{equation*}
Putting the second variations for length and area together, and choosing the parameter $x$ to be constant speed at $s=0$ (so that $|\sigma_x|=f$) we find
\begin{align*}
0&\leq \frac{\partial^2}{\partial s^2}\left(\left|\partial_{\Omega_{t_0}}K_s\right|-f(|K_s|,t_0)\right)\Big|_{s=0}\\
&=\frac{1}{f}\int_0^1\varphi_x^2\,dx-\varphi(0)^2\curv(u_+)-\varphi(1)^2\curv(u_-)\\
&\quad\null-f(f')^2\int_0^1\varphi^2\,dx-f^2f''\left(\int_0^1\varphi\,dx\right)^2.
\end{align*}
This completes the proof of Lemma \ref{lem:second-var-ineq}.
\end{proof}

\begin{lemma}[Time variation inequality]\label{lem:time-var-ineq}
$$
-\frac{\partial f}{\partial t}+f'(\pi-2|K|)+f-f(f')^2\leq \curv(u_-)+\curv(u_+),
$$
where $f'$ and $\frac{\partial f}{\partial t}$ are evaluated at $(|K|,t_0)$, and $f'$ denotes the derivative of $f$ with respect to the first argument.
\end{lemma}

\begin{proof}
Consider any smoothly varying family of regions $\{K_t\}$ for $t\leq t_0$ close to $t_0$, with $K_{t_0}=K$.  Describe the boundary curves by a smooth family of embeddings $\sigma:\ [0,1]\times(t_0-\delta,t_0]\to\RR^2$ with $\sigma(x,t)\in\Omega_t$, $\sigma(0,t)=X(u_+(t),t)$, and $\sigma(1,t)=X(u_-,t)$.  Note that such a family always exists.
Then we have $|\partial_{\Omega_t}K_t|-f(|K_t|,t)\geq 0$ for each $t\in[t_0-\delta,t_0]$, with equality at $t=t_0$.  It follows that $\partial_t\left(|\partial_{\Omega_t}K_t|-f(|K_t|,t)\right)\big|_{t=t_0}\leq 0$.  We compute
$$
|\partial_{\Omega_t}K_t|=\int_0^1|\sigma_x|\,dx,
$$
while 
$$
|K_t|=\frac12\int_0^1\sigma\times\sigma_x\,dx + \frac12\int_{u_-(t)}^{u_+(t)} X\times X_u\,du.
$$
Write $\partial_t\sigma = V+\sigma$.  For convenience we choose the parameter $u$ to be arc-length parametrisation for $t=t_0$.  Differentiating the first equation gives
\begin{align*}
\frac{\partial}{\partial t}|\partial_{\Omega_t}K_t|&=\int_0^1{\bf t}\cdot\partial_x(V+\sigma)\,dx\\
&=|\partial_{\Omega_t}K_t|+\int_0^1{\bf t}\cdot \partial_xV\,dx\\
&=|\partial_{\Omega_t}K_t|+{\bf t}\cdot V\big|_{0}^1+\int_0^1\curv_\sigma {\bf n}\cdot V|\sigma_x|\,dx.
\end{align*}
Since $\sigma(0,t)=X(u_+(t),t)$ and $\sigma(1,t)=X(u_-(t),t)$ for each $t$, we have
\begin{align}
\sigma(0)+V(0) &= X(u_+)-\curv(u_+)\nor(u_+)+\dot u_+\tang(u_+);\label{eq:V0}\\
\sigma(1)+V(1) &= X(u_-)-\curv(u_-)\nor(u_-)+\dot u_-\tang(u_-).\label{eq:V1}
\end{align}
The first terms on left and right cancel.  Since $\nor(u_+)=-{\bf t}(0)$ and $\nor(u_-)={\bf t}(1)$, we have $V(0)\cdot{\bf t}(0)=\curv(u_+)$ and $V(1)\cdot{\bf t}(1)=-\curv(u_-)$, and so
\begin{equation}\label{eq:dtL}
\frac{\partial}{\partial t}|\partial_{\Omega_t}K_t|\Big|_{t=t_0}=
|\partial_{\Omega_{t_0}}K|-\curv(u_-)-\curv(u_+)+f'\int_0^1V\cdot{\bf n}|\sigma_x|\,dx.
\end{equation}
Next we compute the rate of change of the area:
\begin{align*}
\frac{\partial}{\partial t}|K_t|\Big|_{t=t_0}
&=\frac{\partial}{\partial_t}\left(\frac12\int_0^1\sigma\times\sigma_x\,dx + \frac12\int_{u_-(t)}^{u_+(t)} X\times X_u\,du\right)\\
&=\frac12\int_0^1\left[(\sigma+V)\times\sigma_x+\sigma\times\partial_x(\sigma+V)\right]\,dx\\
&\quad\null+\frac12\int_{u_-}^{u_+}\left[(X-\curv\nor)\times X_u + X\times\partial_u\left(X-\curv\nor\right)\right]\,du\\
&\quad\null+ \dot u_+ X(u_+)\times \tang(u_+)-\dot u_-X(u_-)\times\tang(u_-)\\
&=2|K|+\int_0^1V\times{\bf t}|\sigma_x|\,dx+\frac12\sigma\times V\Big|_0^1\\
&\quad\null +\int_{u_-}^{u_+}\curv\,du-\frac{\curv}2X\times\nor\Big|_{u_-}^{u_+}\\
&\quad\null+ \dot u_+ X(u_+)\times \tang(u_+)-\dot u_-X(u_-)\times\tang(u_-)\\
&=2|K|+\int_0^1 V\cdot{\bf n}|\sigma_x|\,dx - \int_{u_-}^{u_+}\curv\,du,
\end{align*}
where in the last step we used equation \eqref{eq:V0} and \eqref{eq:V1}, the identities $\sigma(0)=X(u_+)$, $\sigma(1)=X(u_-)$, ${\bf t}(0)=-\nor(u_+)$, ${\bf t}(1)=\nor(u_-)$, $\tang(u_-)={\bf n}(0)$, and $\tang(u_+)=-{\bf n}(1)$,
and the fact that the parameter $u$ is chosen to be the arc-length parameter at time $t_0$, so that $|X_u|=1$.  Now since $\sigma([0,1],t_0)$ and $X([u_-,u_+],t_0)$ form a simple closed curve with two corners of angle $\pi/2$, the theorem of turning tangents implies
$$
\int_{a_-}^{a_+}\curv\,du + \int_{0}^{1}\curv_\sigma |\sigma_x|\,dx=\pi,
$$
so that (since $|\sigma_x|=f$ and $\curv_\sigma=f'$)
$$
\int_{u_-}^{u_+}\curv\,du = \pi-ff',
$$
and hence
\begin{equation}\label{eq:dtA}
\frac{\partial}{\partial t}|K_t|\Big|_{t=t_0}=2|K|+ff'-\pi.
\end{equation}

Finally, combining equations \eqref{eq:dtL} and \eqref{eq:dtA} we deduce
\begin{align*}
0&\geq \partial_t\left(|\partial_{\Omega_t}K_t|-f(|K_t|,t)\right)\big|_{t=t_0}\\
&=f-\curv(u_-)-\curv(u_+)+f'(\pi-2|K|)-f(f')^2-\frac{\partial f}{\partial t}
\end{align*}
as
claimed.
\end{proof}

Now we can complete the proof of Theorem \ref{thm:ODEcomparison}:   Combining the inequality from Lemma \ref{lem:time-var-ineq} with inequality \eqref{eq:sec.var.F}, we find
$$
-\frac{\partial f}{\partial t}+f+f'(\pi-2|K|)-f(f')^2\leq \curv(u_-)+\curv(u_+)\leq \frac{1}{f}{\mathcal F}(ff',f^3f'')
$$
where $f$, $f'$ and $f''$ are evaluated at $(|K|,t_0)$.  But this contradicts the strict inequality in the theorem.  Therefore the inequality $\Psi(\Omega_t,a)>f(a,t)$ remains true as long as the solution exists.
\end{proof}

\section{The Isoperimetric profile of symmetric convex curves with four vertices}

In this section we determine the isoperimetric regions and isoperimetric profile for convex domains which are symmetric in both coordinate axes and have exactly four vertices.   This result is somewhat analogous to the characterization of isoperimetric regions in rotationally symmetric surfaces with decreasing curvature due to Ritor\'e \cite{Ritore}.  We use it
 in the next section to construct solutions of the differential inequality arising in Theorem \ref{thm:ODEcomparison}.

\begin{theorem}\label{thm:model_isoperim}
 Let $\gamma=\partial\Omega$, where $\Omega$ is a smoothly bounded uniformly convex region of area $\pi$ with exactly four vertices and symmetry in both coordinate axes, with the points of maximum curvature on the $x$ axis.   Let $X: \RR\to\RR^2$ be the map which takes $\theta\in\RR$ to the point in $\gamma$ with outward normal direction $(\cos\theta,\sin\theta)$.  Then for each $\theta\in(0,\pi)$ there exists a unique constant curvature curve $\sigma_\theta$ which is contained in $\Omega$ and has endpoints at $X(\theta)$ and $X(-\theta)$ meeting $\gamma$ orthogonally.  Let $K^x_\theta$ denote the connected component of $\Omega\setminus\sigma_\theta$ containing the vertex of $\gamma$ on the positive $x$ axis.   Then there exists a smooth, increasing diffeomorphism $\theta$ from $(0,\pi)$ to $(0,\pi)$ such that $K_a=K^x_{\theta(a)}$ has area $a$ for each $a\in(0,\pi)$, and the isoperimetric regions of area $a$ in $\Omega$ are precisely $K_{a}$ and its reflection in the $y$ axis.
\end{theorem}

\begin{figure}
\includegraphics[scale=0.7]{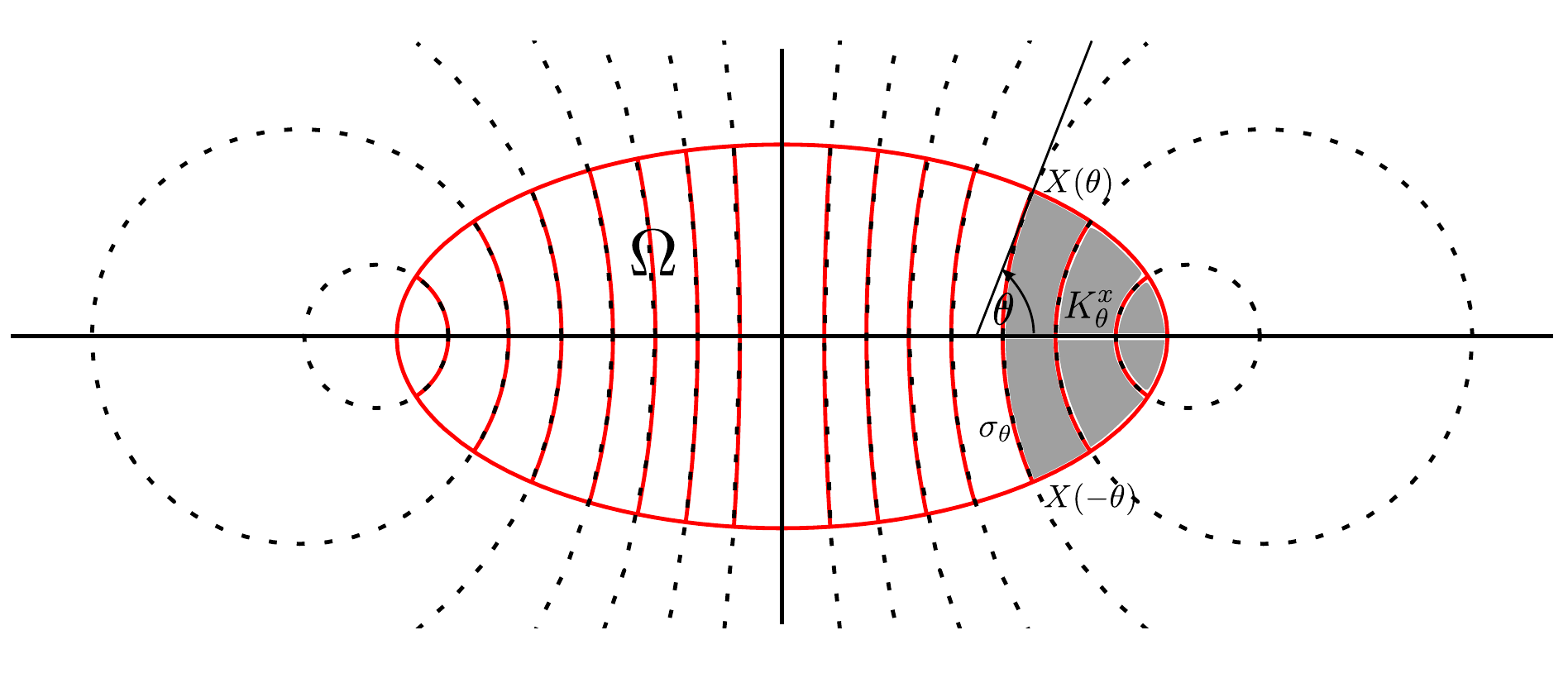}
\caption{Isoperimetric regions of the ellipse $\{x^2+4y^2\leq 4\}$, according to Theorem \ref{thm:model_isoperim}.}
\end{figure}

\begin{proof}
Since $\Omega$ is uniformly convex and $\gamma$ is smooth, for each $\theta\in\RR$ there exists a unique point $X(\theta)\in\gamma$ where the outward unit normal is equal to $e^{i\theta}=(\cos\theta,\sin\theta)$.  Furthermore we can write $X(\theta)$ in terms of the support function $h:\ \RR/(2\pi\ZZ)\to\RR$ of $\Omega$, defined by $h(\theta)=\sup\{\langle x,e^{i\theta}\rangle:\ x\in\Omega\}$:   
\begin{equation}\label{eq:X}
X(\theta)=(h(\theta)+ih'(\theta))\E^{i\theta}.
\end{equation}
The radius of curvature at the corresponding point is then given by $h''+h$.
The symmetry assumptions on $\Omega$ imply that $h$ is even and $\pi$-periodic.

For strictly convex $\Omega$ it was proved by Sternberg and Zumbrun \cite{SZ} that the boundary $\partial_{\Omega}K$ of an isoperimetric region $K$ is connected.  Therefore we have two possibilities:  The first case is where the curvature of the boundary is zero, in which case $K=\Omega\cap\{x:\ \langle x,e^{i\theta}\rangle \leq r\}$ for some $\theta,r\in\RR$.  Since $\partial_{\Omega}K$ meets $\gamma$ orthogonally, the endpoints of points of intersection must have normal orthogonal to $e^{i\theta}$, and so are the two points $X(\theta+\pi/2)$ and $X(\theta-\pi/2)$.  But then we must also have 
$\langle X(\theta+\pi/2),\E^{i\theta}\rangle = \langle X(\theta-\pi/2),\E^{i\theta}\rangle$, which by \eqref{eq:X} and the symmetry of $h$ implies 
\begin{align*}
0&=\left\langle\left(\!h(\theta\!+\!\frac{\pi}{2})+ih'(\theta\!+\!\frac{\pi}{2})\!\right)\E^{i(\theta+\frac{\pi}{2})}\!-\!\left(\!h(\theta\!-\!\frac{\pi}{2})\!+\!ih'(\theta\!-\!\frac{\pi}{2})\!\right)\E^{i(\theta-\frac{\pi}{2})},e^{i\theta}
\right\rangle\\
&=-h'(\theta+\frac{\pi}{2})-h'(\theta-\frac{\pi}{2})\\
&=-2h'(\theta+\frac{\pi}{2}).
\end{align*}

\begin{lemma}\label{lem:htheta}
$h'(\theta)=0$ only for $\theta=\frac{k\pi}{2}$, $k\in\ZZ$.
\end{lemma}

\begin{proof}
Since $h$ is even and $\pi$-periodic, we have $h^{(3)}+h'=0$ at each of the points $\theta=\frac{k\pi}{2}$, so there are four vertices (critical points of curvature, hence of the radius of curvature) at $\theta=0$, $\pi/2$, $\pi$ and $3\pi/2$.  Since there are precisely four vertices by assumption, we have $h^{(3)}+h'\neq 0$ at every other point.  By assumption $h''(\pi/2)+h(\pi/2)>h''(0)+h(0)$, so we must have $h^{(3)}+h'>0$ on $(0,\pi/2)$.

Now let $P(\theta)=h'(\theta)\cos\theta-h''(\theta)\sin\theta$ and $Q(\theta)=h'(\theta)\sin\theta+h''(\theta)\cos\theta$.  We have $P(0)=h'(0)=0$ and $P' = -(h^{(3)}+h')\sin\theta<0$ on $(0,\pi/2)$, so $P<0$ on $(0,\pi/2]$.  Also we have $Q(\pi/2)=h'(\pi/2)=0$ and $Q'=(h^{(3)}+h')\cos\theta>0$ on $(0,\pi/2)$, so $Q<0$ on $[0,\pi/2)$.  But then $h'(\theta) = P(\theta)\cos\theta+Q(\theta)\sin\theta<0$ on $(0,\pi/2)$.  Thus $h$ has no critical points in $(0,\pi/2)$, and hence also no critical points on $\left(\frac{k\pi}{2},\frac{(k+1)\pi}{2}\right)$ for any $k\in\ZZ$ since $h$ is even and periodic.
\end{proof}

It follows that the only possibilities for isoperimetric regions of this kind are the intersections of the coordinate half-spaces with $\Omega$.   These all divide the area of $\Omega$ into regions with area $\pi/2$, and so the only ones which can be isoperimetric are those with shorter length of intersection, which are the halfspaces of positive or negative $x$.

The second case is where the curvature of the boundary of $K$ is non-zero, in which case $K=\Omega\cap B_r(p)$ for some $r>0$ and $p\in\RR^2$.  In this case the intersection of the circle $S_r(p)$ with $\gamma$ consists of two points $X(\theta_2)$ and $X(\theta_1)$, and since the circle meets $\gamma$ orthogonally the line from $p$ to $X(\theta_1)$ is orthogonal to $e^{i\theta_1}$, and we have $p=X(\theta_1)+rie^{i\theta_1}$.  Similarly $p=X(\theta_2)-rie^{i\theta_2}$.   That is, we have by \eqref{eq:X} 
$$
p=(h(\theta_1)+ih'(\theta_1)+ir)e^{i\theta_1}=(h(\theta_2)+ih'(\theta_2)-ir)e^{i\theta_2}.
$$
The equality on the right can be solved for $r$:  Multiply by $e^{-i(\theta_1+\theta_2)/2}$ and write $\Delta=\frac{\theta_2-\theta_1}{2}$.  This gives
\begin{align*}
2ir\cos\Delta &= (h(\theta_2)-h(\theta_1))\cos\Delta-(h'(\theta_2)+h'(\theta_1))\sin\Delta\\
&\quad\null+i\left[(h(\theta_2)+h(\theta_1)\sin\Delta+(h'(\theta_2)-h'(\theta_1))\cos\Delta\right].
\end{align*}
Since $r$ is real, the real part of the right-hand side vanishes.  We denote this by $G(\theta_1,\theta_2)$:
$$
G(\theta_1,\theta_2):= (h(\theta_2)-h(\theta_1))\cos\Delta-(h'(\theta_2)+h'(\theta_1))\sin\Delta.
$$

\begin{lemma}
The zero set of $G$ consists precisely of the points $\{\theta_1+\theta_2=k\pi\}$ for $k\in\ZZ$ and the points $\{\theta_2-\theta_1=2k\pi\}$ for $k\in\ZZ$.
\end{lemma}

\begin{proof}
The symmetry of $h$ implies $h(\theta)=h(\theta+k\pi)=h(k\pi-\theta)$ and $h'(\theta)=h'(\theta+k\pi)=-h'(k\pi-\theta)$ for any $k\in\ZZ$.  Thus when $\theta_2+\theta_1=k\pi$ we have $h(\theta_2)=h(k\pi-\theta_1)=h(\theta_1)$ and $h'(\theta_2)=h'(k\pi-\theta_1)=-h'(\theta_1)$, and hence $G=0$.  Also, when $\theta_2-\theta_1=2k\pi$ then we have $\sin\Delta=0$ and $h(\theta_2)-h(\theta_1)=0$, so $G=0$.  To show the converse, we compute the derivative of $G$ along lines of constant $\theta_1+\theta_2$:  
\begin{align*}
\frac{\partial G}{\partial\theta_2}&=h'(\theta_2)\cos\Delta-\frac12(h(\theta_2)-h(\theta_1))\sin\Delta\\
&\quad\null-h''(\theta_2)\sin\Delta-\frac12(h'(\theta_1)+h'(\theta_2))\cos\Delta\\
&=-(h''(\theta_2)+h(\theta_2))\sin\Delta\\
&\quad\null+\frac12(h'(\theta_2)-h'(\theta_1)\cos\Delta+\frac12(h(\theta_1)+h(\theta_2))\sin\Delta;\\
\frac{\partial G}{\partial\theta_1}&=-h'(\theta_1)\cos\Delta+\frac12(h'(\theta_1)-h'(\theta_2)\sin\Delta\\
&\quad\null-h''(\theta_1)\sin\Delta+\frac12(h'(\theta_2)+h'(\theta_1))\cos\Delta\\
&=-(h''(\theta_1)+h(\theta_1))\sin\Delta\\
&\quad\null+\frac12(h'(\theta_2)-h'(\theta_1)\cos\Delta+\frac12(h(\theta_1)+h(\theta_2))\sin\Delta.
\end{align*}
Taking the difference gives
\begin{equation}\label{eq:derivG}
\frac{\partial G}{\partial\theta_2}-\frac{\partial G}{\partial\theta_1} = [(h''(\theta_1)+h(\theta_1))-(h''(\theta_2)+h(\theta_2)]\sin\Delta.
\end{equation}
As above, the assumption that $\gamma$ has exactly four vertices with the points of maximum curvature on the $x$ axis implies that $h''+h$ is strictly increasing on intervals $[k\pi,(k+\frac12)\pi]$, and strictly decreasing on intervals $[(k+\frac12)\pi,(k+1)\pi]$ for any $k\in\ZZ$.   The 
symmetries of $h$ imply that $G$ is odd under reflection in the lines $\theta_1+\theta_2=0$, $\theta_2-\theta_1=0$ and $\theta_2+\theta_1=\pi$, and even under reflection in the line $\theta_2-\theta_1=\pi$, and that $G(\theta_1+\pi,\theta_2+\pi)=G(\theta_1,\theta_2)$ and $G(\theta_1+\pi,\theta_2-\pi)=-G(\theta_1,\theta_2)$.  Therefore it suffices to show that $G\neq 0$ on the fundamental domain $W=\left\{(\theta_1,\theta_2):\ \theta_1\in(-\frac{\pi}{2},\frac{\pi}{2}),\ \theta_2\in(|\theta_1|,\pi-|\theta_1|)\right\}$.  The monotonicity of $h''+h$ implies that $h''(\theta_2)+h(\theta_2)> h''(\theta_1)+h(\theta_1)$ on $W$.   Equation \eqref{eq:derivG} implies that $G$ is increasing along lines of constant $\theta_1+\theta_2$ in $W$ away from the line $\{\theta_2=\theta_1\}$ where $G=0$.  Hence $G$ is positive on $W$ as required.
\end{proof}
The lemma implies that the only candidates for boundaries of isoperimetric regions of this type are the 
following two families:  

For each $\theta\in(0,\pi/2)$ there is a unique region 
$K^x_\theta=\Omega\cap B_{r(\theta)}(p(\theta))$, where $p(\theta)$ lies in the positive $x$ axis, and the outward normals to $\Omega$ at the endpoints of $\partial_\Omega K^x_\theta$ make angles $\pm\theta$ with the positive $x$ axis.  In this family we also take $K^x_{\pi/2}$ to be the intersection of $\Omega$ with the positive $x$ half-space, and $K^x_{\pi-\theta}$ is the exterior in $\Omega$ of the reflection of $K^x_\theta$ in the $y$ axis.

The second family is similar but with centres on the $y$ axis:  $K^y_\theta=\Omega\cap B_{\rho(\theta)}(q(\theta))$, where $q(\theta)$ lies in the positive $y$ axis, and the outward normals to $\Omega$ at the endpoints of $\partial_\Omega K^y_{\theta}$ makes angles $\pm\theta$ with the positive $y$ axis, for $0<\theta<\pi/2$, while $K^y_{\pi/2}$ is the intersection of $\Omega$ with the upper $y$ half-space, and $K^y_{\pi-\theta}$ is the exterior in $\Omega$ of the reflection  of $K^y_\theta$ in the $x$ axis.  Note that these regions are candidates for the isoperimetric region only if $K^y_\theta$ has only a single boundary curve, which is not always the case.

Note that we do not claim at this stage that the regions $K^x_\theta$ and $K^y_\theta$ define simply connected sub-regions of $\Omega$ for every $\theta\in(0,\pi)$:  The curves certainly exist, but may intersect the boundary of $\Omega$ at other points.   Indeed this certainly occurs for very long, thin regions for the family $K^y_\theta$.  We will prove below that the family $K^x_\theta$ are always simply connected and have a single boundary component.

The following result shows that only the $K^x_\theta$ can be isoperimetric regions:

\begin{proposition}\label{prop:unstable_y_family}
For any $\theta\in(0,\pi)$ for which $\partial_\Omega K^y_\theta$ is connected, there exists a smoothly family of regions $\{\tilde K(s):\ |s|<\delta\}$ with $\tilde K(0)=K^y_\theta$, $\frac{d}{ds}|K(s)|=0$ for all $s$, and $\frac{d}{ds}|\partial_\Omega\tilde K(s)\big|_{s=0}=0$, and $\frac{d^2}{ds^2}|\partial_\Omega\tilde K(s)\big|_{s=0}<0$.  In particular, $K^y_\theta$ does not minimize length among regions with the same area.
\end{proposition}

\begin{proof}
The idea of the proof is to use the fact that the isoperimetric domains inside a round ball are neutrally stable (with the direction of neutral stability given by rotation around the disk).  We will transplant this variation onto $\partial_\Omega K^y_\theta$ to produce an area-preserving variation for which the second variation of the length $|\partial_\Omega K|$ is negative. 

As in Lemma \ref{lem:first-var} we parametrize $\partial_{\Omega}K^y_\theta$ 
by a smooth map $\sigma_0:\ [0,1]\to\Omega$ with $\sigma_0(0)=X(\pi/2+\theta)$ and $\sigma_0(1)=X(\pi/2-\theta)$, and $|\partial_x\sigma_0|$ constant (equal to the length $|\partial_{\Omega}K^y_\theta|$).
For any smooth function $\varphi:\ [0,1]\to\RR$ with $\int_0^1\varphi\,dx=0$, $\sigma_0$ can be extended to a smooth family of embeddings $\sigma:\ [0,1]\times(\delta,\delta)\to\Omega$ with the following properties:  $\sigma(x,0)=\sigma_0(x)$ for all $x\in[0,1]$; $\sigma(0,s)=X(\theta_+(s))$ and $\sigma(1,s)=X(\theta_-(s))$ for some $\theta_{\pm}(s)$; $\frac{\partial}{\partial s}\sigma(x,s)\big|_{s=0}=\varphi(x){\bf n}(x)$, where ${\bf n}$ is the outward-pointing unit normal to $K^y_\theta$; and the areas of the enclosed regions $K_s$ are constant:
$$
|K_s| = \frac12\int_0^1\sigma\times\sigma_x\,dx+\int_{\theta_-(s)}^{\theta_+(s)}X\times X_\theta\,d\theta=|K^y_\theta|.
$$ 
As in Lemma \ref{lem:first-var} we write $\frac{\partial\sigma}{\partial s}=\eta{\bf n}+\xi{\bf t}$, so that $\eta(x,0)=\varphi(x)$ and $\xi(x,0)=0$.  The computation of Lemma \ref{lem:second-var-ineq} yields the following:
\begin{align*}
\frac{\partial^2}{\partial s^2}|K_s|\big|_{s=0} &= \int_0^1 (\dot\eta+\varphi^2\kappa_\sigma)|\sigma_x|\,dx = 0;\\
\frac{\partial^2}{\partial s^2}|\partial_{\Omega}K_s|\big|_{s=0}
&=\int_0^1\frac{(\varphi_x^2)}{|\sigma_x|}\,dx + \kappa_\sigma\int_0^1\dot\eta|\sigma_x|\,dx
-\varphi(0)^2\kappa(\theta_+)-\varphi(1)^2\kappa(\theta_-).
\end{align*}
The first identity gives an expression for $\int_0^1\dot\eta|\sigma_x|\,dx$, which we substitute in the second equation to give
\begin{equation}\label{eq:variation.in.Omega}
\frac{\partial^2}{\partial s^2}|\partial_{\Omega}K_s|\big|_{s=0}=
\int_0^1\frac{(\varphi_x^2)}{|\sigma_x|} - \kappa_\sigma^2\varphi^2|\sigma_x|\,dx
-\varphi(0)^2\kappa(\pi/2+\theta)-\varphi(1)^2\kappa(\pi/2-\theta),
\end{equation}
since $\kappa_+(0)=\pi/2+\theta$ and $\kappa_-(0)=\pi/2-\theta$.

\begin{figure}
\includegraphics[scale=0.65]{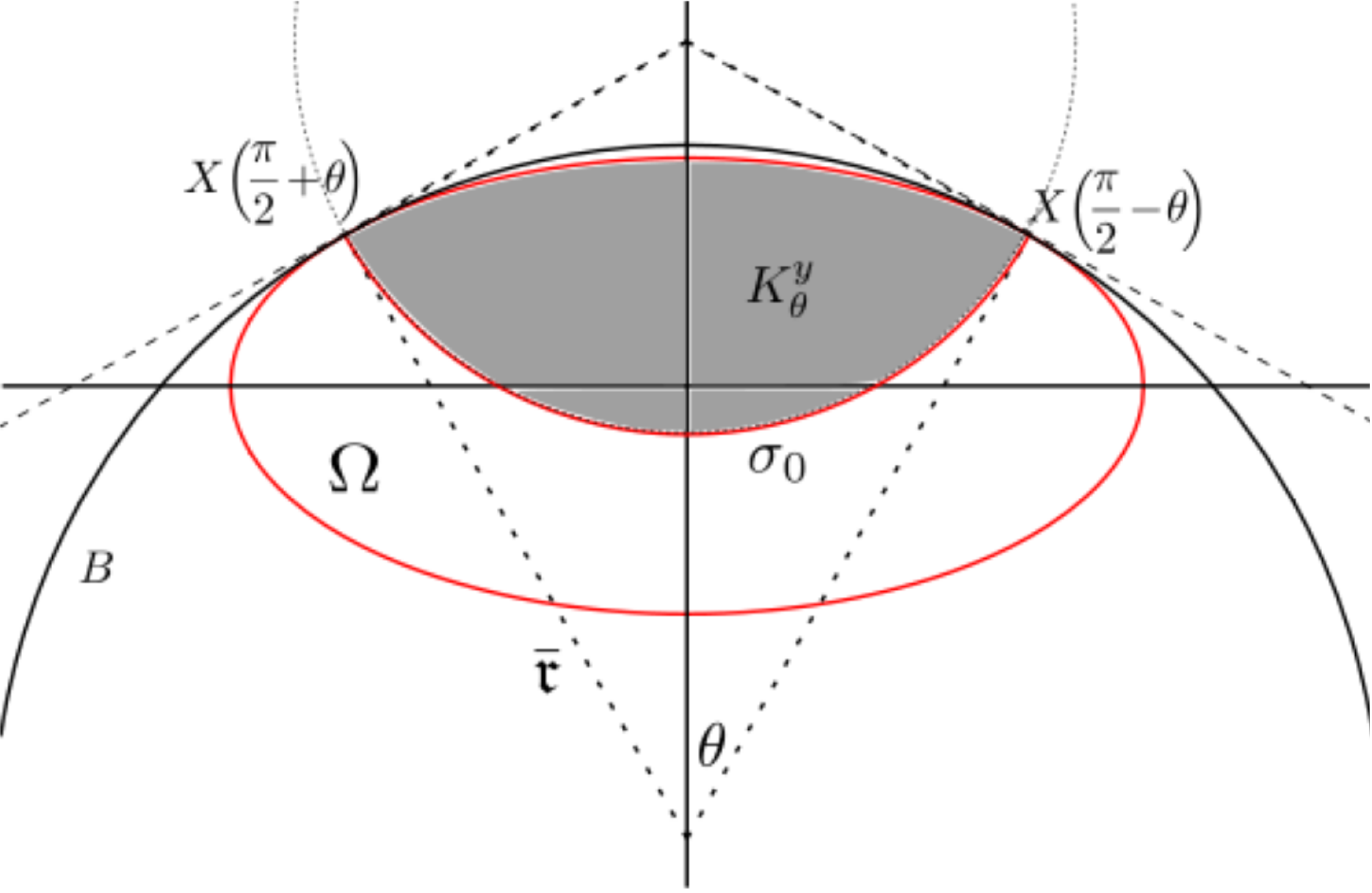}
\caption{A candidate isoperimetric region $K^y_\theta$, given by the intersection with $\Omega$ of a disk with centre on the $y$ axis.  Also shown is a disk $B$ of radius $\bar r=\frac{1}{\bar\curv}$ which meets the same curve orthogonally.}\label{fig:Ky}
\end{figure}

It remains to choose $\varphi$ to make this expression negative.  To do this we note that there is a unique disk $B$ which meets the curve $\sigma_0$ orthogonally at the same pair of endpoints.  By symmetry $B$ has centre on the $y$ axis, and we denote the curvature of $B$ by $\bar\curv$.  Now consider the area-preserving variation corresponding to rotation of the curve $\sigma_0$ about the centre of the circle $B$.  This does not change either the enclosed area or the length in $B$, so for the corresponding function $\varphi$ we have
$$
0=
\int_0^1\frac{(\varphi_x^2)}{|\sigma_x|} - \kappa_\sigma^2\varphi^2|\sigma_x|\,dx
-\varphi(0)^2\bar\curv-\varphi(1)^2\bar\curv.
$$
Substituting this in equation \eqref{eq:variation.in.Omega} then gives a variation in $\Omega$ for which
\begin{align*}
\frac{\partial^2}{\partial s^2}|\partial_{\Omega}K_s|\big|_{s=0}&=\varphi(0)^2\left(\bar\curv-\kappa(\pi/2+\theta\right)+\varphi(1)^2\left(\bar\curv-\curv(\pi/2-\theta)\right)\\
&=2\varphi(0)^2\left(\bar\curv-\kappa(\pi/2+\theta\right)).
\end{align*}
where we used the symmetry in the last equality.  Since $\varphi(0)\neq 0$, it remains only to prove that $\kappa(\pi/2+\theta)>\bar\curv$.

By symmetry it suffices to prove this for $0<\theta\leq\pi/2$.  The point on $\gamma$ with normal direction making angle $\theta$ with the $y$ axis is given by $X(\theta+\pi/2)$, where $X$ is given by equation \eqref{eq:X}.
Note that $\frac{\partial X}{\partial\theta}=(h''+h)i\E^{i\theta}=\rr i\E^{i\theta}$, so integrating we find
$$
X(\pi/2+\theta)=X(\pi/2)+\int_{\pi/2}^{\pi/2+\theta}\rr i\E^{i\theta'}d\theta'.
$$
By symmetry, the $x$ component of $X(\pi/2)$ vanishes, so 
$$
\langle X(\pi/2+\theta),1\rangle=-\int_{\pi/2}^{\pi/2+\theta}\rr \sin(\theta')\,d\theta'.
$$
Now we do the same computation for the circle which meets both $X(\pi/2+\theta)$ and $X(\pi/2-\theta)$ tangentially (i.e. for the boundary of $B$).  Denote the point on this circle with normal direction $\theta$ by $\bar X(\theta)$. By symmetry we have $\bar X(\pi/2)$ on the $y$ axis, and hence the $x$ component of $\bar X(\pi/2+\theta)$ is given by
$$
\langle \bar X(\pi/2+\theta),1\rangle = -\int_{\pi/2}^{\pi/2+\theta}\bar{\rr}\sin(\theta')\,d\theta',
$$
where $\bar\rr$ is the radius of curvature of this circle.
Since $X(\pi/2+\theta)=\bar X(\pi/2+\theta)$, we have
$$
\bar\rr =\frac{\int_{\pi/2}^{\pi/2+\theta}\rr(\theta')\sin(\theta')\,d\theta'}{\int_{\pi/2}^{\pi/2+\theta}\sin(\theta')\,d\theta'}.
$$
By assumption, $\rr(\theta')$ is strictly decreasing on the interval $[\pi/2,\pi/2+\theta]$, so $\rr(\theta')>\rr(\pi/2+\theta)$ for every $\theta'\in[\pi/2,\pi/2+\theta)$.  Therefore we have 
$$
\frac{1}{\bar\curv}=\bar\rr>\rr(\pi/2+\theta)=\frac{1}{\curv(\pi/2+\theta)}
$$
as required.
This completes the proof of Proposition \ref{prop:unstable_y_family}.
\end{proof}

To complete the proof of Theorem \ref{thm:model_isoperim} it remains to check that $K^x_\theta$ has a single boundary curve in $\Omega$ for each $\theta\in(0,\pi)$, and that
for each value of $a\in(0,\pi)$ there is a unique $\theta\in(0,\pi)$ such that $|K^x_\theta|=a$.   This suffices to prove the Theorem, since the result of \cite{SZ} implies that the isoperimetric region is connected and simply connected, and hence must consist either of one of the regions $K^x_\theta$ or the exterior of such a region.

\begin{lemma}\label{lem:inscribed.disk}
For each $\theta\in(0,\pi)$ the disc $B$ centred on the $x$ axis which passes through $X(\theta)$ and $X(-\theta)$ has curvature strictly greater than the curvature of $\gamma$ at $X(\pm\theta)$, and is contained in $\Omega$.
\end{lemma}

\begin{proof}
We first show the inequality between the curvatures.  
By assumption, the point of maximum curvature (hence minimum $\rr$) is at $\theta=0$, and we have $\rr$ strictly increasing on the interval $(0,\pi/2)$.
Choose the origin to be at the centre $c$ of the ball $B$, and let $h$ be the support function.
From equation \eqref{eq:X} we have $X'(\phi) = i\rr\E^{i\phi}$, so the vertical component $y$ satisfies $y'(\phi) = \rr(\phi)\cos\phi$.  Since $y(0)=0$ by symmetry, we have $y(\theta) = \int_0^\theta \rr(\phi)\cos\phi\,d\phi<\rr(\theta)\int_0^\theta\cos\phi\,d\phi$.  Now the ball $B$ also has $y$ coordinate $\bar y(0)=0$ and $\bar y'(\phi) = \bar{\rr}\cos\phi$, and by assumption $\bar y(\theta)=y(\theta)$, so we have
$$
\bar{\rr}\int_0^\theta\cos\phi\,d\phi = \bar y(\theta)=y(\theta)<\rr(\theta)\int_0^\theta\cos\phi\,d\phi,
$$
from which it follows that $\rr(\theta)>\bar{\rr}$.

Next we show that the ball $B$ is inscribed.  We prove this only for $\theta\in(0,\pi/2)$, since the result for $\theta>\pi/2$ follows by symmetry, and for $\theta=\pi/2$ by continuity.
It suffices to show that $h\geq\bar{\rr}$ everywhere. We prove this first on the interval $[0,\theta]$:
Set $v=h'$, and $q=\rr'>0$.  From equation \eqref{eq:X} we note that $X(0) = (h(0),h'(0))$ lies on the $x$ axis, so $v(0)=h'(0)=0$.  Also, by our choice of origin ${\bar\rr}\E^{i\theta}=X(\theta)=h(\theta)\E^{i\theta}+ih'(\theta)\E^{i\theta}$, so $v(\theta)=h'(\theta)=0$ and $h(\theta)=\bar{\rr}$.  We can also write $v''+v=q>0$.   It follows that $v<0$ on $[0,\theta]$:  For example we can use the representation formula
$$
v(\phi) = -\frac{\sin\phi}{\sin\theta}\int_{\phi}^\theta\sin(\theta-\alpha)\,d\alpha
-\frac{\sin(\theta-\phi)}{\sin\theta}\int_0^\phi\sin\alpha\,d\alpha<0
$$
for $0<\phi<\theta$.  Therefore we have $h(\phi)=h(\theta)-\int_\phi^\theta h'(\alpha)\,d\alpha>h(\theta)=\bar{\rr}$ for $0\leq \phi<\theta$.  By symmetry the same holds for $-\theta<\phi\leq 0$.

Now on the interval $(\theta,\pi/2]$ we have $\rr(\phi)>\rr(\theta)$, so the function $w=h-\bar{\rr}$ satisfies 
$w(0)=0$, $w'(0)=0$ and $f=w''+w>0$.  Therefore 
$$
w(\phi) = \int_\theta^\phi\sin(\phi-\alpha)f(\alpha)\,d\alpha>0,
$$
so that $h(\phi) = w(\phi)+\bar{\rr}>\bar{\rr}$ for $\theta<\phi\leq\pi/2$, and by symmetry we now have
$h\geq\bar{\rr}$ on $[-\pi/2,\pi/2]$, with a strict inequality except at $\pm\theta$.  Also, we have
$$
w'(\phi) = \int_\theta^\phi\cos(\phi-\alpha)f(\alpha)\,d\alpha>0,
$$
Thus in particular $x(\pi/2)=-w'(\pi/2)<0$.  The reflection symmetry implies that $y(\pi-\phi) = y(\phi)$ and
$x(\pi-\phi)-x(\pi/2)=-x(\phi)-x(\pi/2)$, so $x(\pi-\phi)=-x(\phi)+2x(\pi/2)<-x(\phi)$.  Finally, for $\phi\in(-\pi/2,\pi/2)$ we have 
\begin{align*}
h(\pi+\phi) 
&= x(\pi+\phi)\cos(\pi+\phi)+y(\pi+\phi)\sin(\pi+\phi)\\
&= -(2x(\pi/2)-x(-\phi))\cos\phi+y(\phi)\sin\phi\\
&=h(\phi)-2x(\pi/2)\cos\phi\\
&>\bar{\rr}.
\end{align*}
Thus we have $h\geq{\bar\rr}$ everywhere, so the ball $B$ is inscribed in $\Omega$.
\end{proof}

It follows that the boundary $\partial_\Omega K^x_\theta$ consists of a single arc from $X(\theta)$ to $X(-\theta)$, since two circles cannot meet at three points unless they are identical.  It remains only to show that the area is monotone along this family.  

\begin{figure}
\includegraphics[scale=0.65]{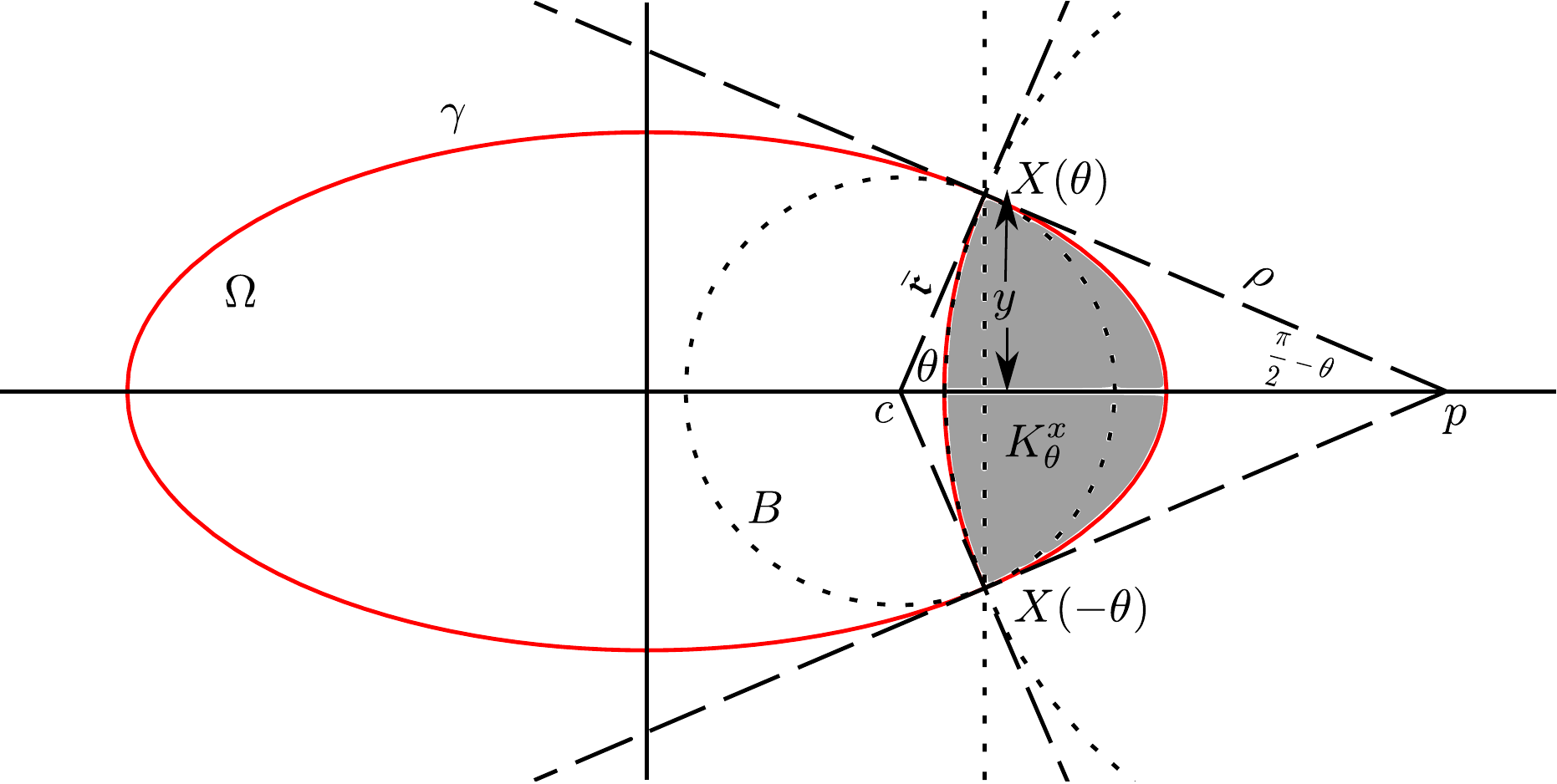}
\caption{Construction of the region $K^x_\theta$ by intersecting $\Omega$ with a disk of radius $\rho$ centred at $p$, showing the inscribed disk $B$ of radius $\bar{\rr}$. }\label{fig:disk.construction}
\end{figure}

We assume initially that $\theta\in(0,\pi/2)$.  Then
the radius of curvature $\rho$ of the boundary curve of $K^x_\theta$ is given by $\rho=\frac{y}{\cos\theta}$, where $y=\langle X(\theta),i\rangle$ is the distance of $X(\theta)$ from the $x$ axis.  Noting that $\partial_\theta X = i\rr\E^{i\theta}$, we have $\partial_\theta y = \langle i\rr\E^{i\theta},i\rangle = \rr\cos\theta$, where $\rr$ is the radius of curvature of $\gamma$ at $X(\theta)$.  From this we obtain the following expression for the rate of change of the radius of curvature $\rho$ of the boundary as $\theta$ varies:
$$
\partial_\theta\rho = \partial_\theta\left(\frac{y}{\cos\theta}\right) = \frac{\rr\cos\theta}{\cos\theta}+\frac{y\sin\theta}{\cos^2\theta} = \rr+\rho\tan\theta.
$$

An expression for the area of $K^x_\theta$ can be computed as follows:  We compute the area of the sector of the disk of radius $\rho$ and angle $\pi-2\theta$, subtract the area of the triangle subtended by $p$, $X(\theta)$ and $X(-\theta)$, and add the area between $\gamma$ and the line from $X(\theta)$ to $X(-\theta)$:  This gives (assuming $\theta\in(0,\pi/2)$)
$$
|K^x_\theta| = \left(\frac\pi2-\theta\right)\rho^2-\rho^2\sin\theta\cos\theta+\int_0^\theta (X(\theta')-X(-\theta'))\times X_\theta(\theta')\,d\theta'.
$$
Differentiating with respect to $\theta$, we find:
\begin{align*}
\partial_\theta\left|K^x_\theta\right|&=-\rho^2+(\pi-2\theta)\rho(\rr+\rho\tan\theta)-\rho^2(\cos^2\theta-\sin^2\theta)\\
&\quad\null-2\rho\sin\theta\cos\theta(\rr+\rho\tan\theta)
+ (X(\theta)-X(-\theta))\times\rr i\E^{i\theta}\\
&=\rho^2\left(-2+(\pi-2\theta)\tan\theta\right)+\rr\rho\left((\pi-2\theta)-2\sin\theta\cos\theta\right)\\
&\quad\null +2\rr\left[\begin{array}{c}0 \\y\end{array}\right]\times \left[\begin{array}{c}-\sin\theta \\\cos\theta\end{array}\right]\\
&=\rho^2(-2+(\pi-2\theta)\tan\theta)+\rr\rho(\pi-2\theta).
\end{align*}
Now we use the result of Lemma \ref{lem:inscribed.disk} which gives $\rr>\bar\rr=\frac{\rho}{\tan\theta}$, so that
\begin{align*}
\partial_\theta\left|K^x_\theta\right|&>\rho^2\left(-2+(\pi-2\theta)\left(\tan\theta+\frac{1}{\tan\theta}\right)\right).\\
&=\rho^2\left(-2+\frac{\pi-2\theta}{\sin\theta\cos\theta}\right)\\
&=\frac{2L^2}{z^2\sin z}\left(z-\sin z\right),
\end{align*}
where $z=\pi-2\theta$ and $L=|\partial_\Omega K^x_\theta|$ is the length of the boundary curve, and we used the identity $z\rho=L$. The right-hand side is strictly positive for $z\in(0,\pi)$, and has limit $L^2/3$ as $z\to 0$.  It follows that $\partial_\theta A$ is strictly positive for $\theta\in(0,\pi/2]$, and by symmetry the same is true for $\theta\in[\pi/2,\pi)$.
\end{proof}

\begin{remark}
Although we do not need it here, one can prove that the family $K^x_\theta$ is increasing in $\theta$, and in fact one can construct a smooth embedding $\sigma$ from $(0,1)\times(0,\pi)$ to the interior of $\Omega$ such that $K^x_\theta=\sigma((0,1)\times(0,\theta)$ and $\partial_\theta\sigma = \eta{\bf n}$, so that $\sigma$ varies in the normal direction everywhere.
\end{remark}

\section{The equality case and model solutions}\label{sec:equality}

In this section we demonstrate a correspondence between solutions of the comparison equation 
arising in Theorem \ref{thm:ODEcomparison}, 
\begin{equation}\label{eq:comparison.equation}
\frac{\partial f}{\partial t}=-f^{-1}{\mathcal F}[ff',f^3f'']+f+f'(\pi-2a)-f(f')^2
\end{equation}
and certain solutions of the normalized curve-shortening flow.  Note that by the expression \eqref{eq:expr.F}, equation \eqref{eq:comparison.equation} is a strictly parabolic fully nonlinear equation for $f$ in the region where ${\mathcal F}[ff',f^3f'']>0$.  

Most important for our purposes is the following method of constructing solutions:

\begin{theorem}\label{thm:model.to.solution}
Let $\Omega_0$ be a compact convex subset of $\RR^2$, symmetric in both coordinate axes and with smooth boundary curve $\gamma_0$ given by the image of a smooth embedding $X_0:\ S^1\to\RR^2$ and having exactly four vertices, with the maxima of curvature located on the $x$ axis.  
Let $X:\ S^1\times[0,T)\to\RR^2$ be the solution of \eqref{eq:NCSF} with initial data $X_0$.  Then for each $t\in[0,T)$, the region $\Omega_t$ enclosed by $\gamma_t=X(S^1,t)$ is a compact convex region symmetric in both coordinate axes, with exactly four vertices and with the maxima of curvature located on the $x$ axis.   For each $t$, let $K_{a,t}$ be the family of isoperimetric regions for $\Omega_t$ constructed in Theorem \ref{thm:model_isoperim}, and define $f(a,t)=\left|\partial_{\Omega_t}K_{a,t}\right|$.  Then $f:\ (0,\pi)\times[0,T)\to\RR$ is a symmetric concave solution of the equation \eqref{eq:comparison.equation}
with $\lim_{a\to 0}\frac{f(a,t)}{\sqrt{2\pi a}}=1$ and ${\mathcal F}[ff',f^3f'']>0$.
\end{theorem}

\begin{proof}
The symmetry of $\Omega_t$
follows from the geometric invariance and uniquess of solutions, and preservation of convexity was proved in \cite{GH}.  The result of \cite{Ang.zero.count} implies that the number of critical points of curvature cannot increase, and the four-vertex theorem implies there are always at least four vertices, so there are always exactly four vertices for $t>0$.  The symmetry implies that these are located on the axes, and the maxima of curvature therefore remain on the $x$ axis.  It follows from Theorem \ref{thm:model_isoperim} that $f(a,t)$ is the isoperimetric profile of $\Omega_t$ for each $t$.  The symmetry of $f$ is immediate from the symmetry of $\Omega_t$ and the definition of $f$ (i.e. we have $f(a,t)=f(\pi-a,t)$).
The concavity of $f$ is proved in \cite{SZ} (in fact it was proved in \cite{Kuwert} that $f^2$ is also concave --- this can be deduced directly by substituting $\varphi=1$ in the second variation inequality \eqref{eq:sec.var.ineq.model} below and using the convexity of $\Omega_t$).
 It remains to show that  $f$ satisfies equation \eqref{eq:comparison.equation}.

For any fixed $t$, along the family $\{K_{a,t}\}$ we have $|\partial_{\Omega_t}K_{a,t}|=f(|K_{a,t}|,t)$, while for all regions we have $|\partial_{\Omega_t}K|\geq f(|K|,t)$.  It follows from Lemma \ref{lem:first-var} that $\curv_\sigma=f'$, where $\sigma$ is the curvature of the boundary curve $\sigma$ of $K_{a,t}$.  By Lemma \ref{lem:second-var-ineq} the second variation inequality holds, i.e.
\begin{equation}\label{eq:sec.var.ineq.model}
\curv(u_-)\varphi(1)^2+\curv(u_+)\varphi(0)^2\leq \frac{1}{f}\int_0^1\varphi_x^2\,dx-f(f')^2\int_0^1
\varphi^2\,dx -f^2f''\left(\int_0^1\varphi\,dx\right)^2.
\end{equation}
On the other hand, for the particular choice of $\varphi$ corresponding to moving through the family $\{K_{a,t}\}$ in such a way that the endpoints of the boundary curve move with unit speed, we have equality in  the above inequality, and $\varphi(1)=\varphi(0)=1$.  Therefore by the definition of ${\mathcal F}$, 
\begin{equation}\label{eq:sec.var.equal}
\curv(u_-)+\curv(u_+)=\frac{1}{f}{\mathcal F}(ff',f^3f'').
\end{equation}
Now consider the family of regions $\{K_{a,t}\}$ for fixed $a$, as $t$ varies.  The proof of Lemma \ref{lem:time-var-ineq} gives that
\begin{align}
0&= \partial_t\left(|\partial_{\Omega_t}K_t|-f(|K_t|,t)\right)\big|_{t=t_0}\notag\\
&=f-\curv(u_-)-\curv(u_+)+f'(\pi-2|K|)-f(f')^2-\frac{\partial f}{\partial t}\label{eq:time.var.equal}
\end{align}
Combining equations \eqref{eq:sec.var.equal} and \eqref{eq:time.var.equal}, we deduce that 
\eqref{eq:comparison.equation} holds.
\end{proof}

\begin{corollary}\label{cor:CSF.comparison}
Let $\{\Omega_t:\ 0\leq t<T\}$ be any smooth compact embedded solution of the normalized curve shortening flow \eqref{eq:NCSF}, and let $\{\Theta_t:\ 0\leq t<T\}$ be any solution of \eqref{eq:NCSF} for which $\Theta_0$ is a smoothly bounded compact convex region with reflection symmetries in both coordinate axes and exactly four vertices, such that $\Psi(\Omega_0,a)\geq\Psi(\Theta_0,a)$ for every $a\in(0,\pi)$.  Then
$\Psi(\Omega_t,a)\geq\Psi(\Theta_t,a)$ for all $a\in(0,\pi)$ and all $t\in[0,T)$.
\end{corollary}

\begin{proof}
Let $f:\ [0,\pi]\times[0,T)\to\RR$ be as in Theorem \ref{thm:model.to.solution}.  Under the assumption $\Psi(a,0)\geq f(a,0)$, we will construct a family of functions $f_\eps$ satisfying the assumptions of Theorem \ref{thm:ODEcomparison} such that $\lim_{\eps\to 0}f_\eps=f$.  That is, we need $f_\eps(a,0)<f(a,0)$, $\limsup_{a\to 0}\frac{f_\eps(a,t)}{\sqrt{2\pi a}}<1$, and $f_\eps$ should satisfy the strict differential inequality in Theorem \ref{thm:ODEcomparison}.

It is convenient to work with the function $v(a,t)=\frac12 f(a,t)^2$ instead of $f$.  Equation 
\eqref{eq:comparison.equation} then becomes 
$$
\frac{\partial v}{\partial t} = {\mathcal G}[v]+2v+v'(\pi-2a)-(v')^2,
$$
where
$$
{\mathcal G}[v] = -{\mathcal F}[ff',f^3f''] = \left(\min\left\{0,\frac{1}{2vv''}-\frac{1}{(v')^2}+\frac{\cos(v'/2)}{2v'\sin(v'/2)}\right\}\right)^{-1}.
$$
Furthermore we know that $v$ is strictly concave by the result of \cite{Kuwert}, and has $|v'(a)|<\pi$ for $a\in(0,\pi)$ by combining the strict concavity with the result of Proposition \ref{prop:asympt.isoperim.profile}.

We accomplish the construction in two stages:  First, we construct strictly concave solutions of the strict differential inequality on slightly smaller domains:
Fix $C>2$, and set $\mu=1-\eps\E^{Ct}$ and $\tau = \int_0^t\mu^{-1}(t')\,dt'$, and define
$$
v_\eps(a,t) = \mu v\left(\pi/2+\mu^{-1}(a-\pi/2),\tau\right),
$$
for $\eps\E^{Ct}\leq a\leq \pi-\eps\E^{Ct}$ and $\eps\E^{Ct}<1$.  Then $v_\eps'=v'$ and $v_\eps v_\eps'' = vv''$, so ${\mathcal G}[v_\eps]={\mathcal G}[v]$.  We also have (denoting time derivatives by dots)
\begin{align*}
\frac{\partial}{\partial t}v_\eps
&=\dot\mu v+\mu\dot\tau\frac{\partial v}{\partial t}-\mu v'\mu^{-2}\dot\mu(a-\pi/2)\\
&={\mathcal G}[v]+(2+\dot\mu)v+v'(\pi-2a)(\mu^{-1}+\frac12\mu^{-1}\dot\mu)-(v')^2\\
&={\mathcal G}[v_\eps]+\frac{2+\dot\mu}{2\mu}(2v_\eps +v_\eps'(\pi-2a))-(v_\eps')^2\\
&<{\mathcal G}[v_\eps]+2v_\eps +v_\eps'(\pi-2a)-(v_\eps')^2
\end{align*}
where $v_\eps$ is always evaluated at $(a,t)$, while $v$ is evaluated at $(\pi/2+\mu^{-1}(a-\pi/2),\tau)$.
We used the identities $\mu\dot\tau=1$ and $\frac{2+\dot\mu}{2\mu}<1$ (coming from our choice $C>2$).
Thus for any $\eps>0$, $v_\eps$ satisfies the required strict inequality.  

Next we must overcome the difficulty caused by the fact that $v_\eps$ is not defined on the whole interval $(0,\pi)$.  To do this we simply replace $v_\eps$ by the smallest concave positive function which lies above it, as follows:  We define
\begin{align*}
\tilde v_\eps(a,t) &= \max\left\{\sup\left\{\frac{a}{x}v_{\eps}(x,t):\ x\in(a,\pi-\eps\E^{Ct})\right\},\right.\\
&\quad\quad\qquad\left.\sup\left\{\frac{\pi-a}{\pi-x}v_\eps(x,t):\ x\in(\eps\E^{Ct},a)\right\}\right\}.
\end{align*}
By smoothness and strict concavity of $v_\eps$, there exists $\eps\E^{Ct}<a_-(t)<\pi/2$ depending smoothly on $t$ such that
$$
\tilde v_\eps(a,t)=\begin{cases}
      \frac{a}{a_-}v_\eps(a_-,t),& 0\leq a\leq a_-; \\
      v_\eps(a),& a_-\leq a\leq \pi-a_-;\\
      \frac{\pi-a}{a_-}v_\eps(a_-,t),& \pi-a_-\leq a\leq\pi,
\end{cases}
$$
where $a_-$ is characterized by the condition $v_\eps'(a_-) = \frac{v_\eps(a_-)}{a_-}$.  $\tilde v_\eps$ is then $C^{1,1}$ and concave, and positive on $(0,\pi)$.  The corresponding function $\tilde f_\eps = \sqrt{2\tilde v_\eps}$ is strictly concave.  Note also that $\tilde v_\eps'(0)=v_\eps'(a_-)\in(0,\pi)$, so the boundary requirement $\limsup_{a\to 0}\frac{\tilde v_\eps(a,t)}{\pi a}<1$ is satisfied.  We check that $\tilde v_\eps$ still satisfies the strict differential inequality:  For $a\in(a_-,\pi-a_-)$ this is immediate since we have checked the inequality for $v_\eps$.  In the case $a\in(0,a_-)$ we have
\begin{align*}
\frac{\partial}{\partial t}\tilde v_\eps(a) &= \frac{a}{a_-}\frac{\partial}{\partial t}v_\eps(a_-)\\
&<\frac{a}{a_-}\left({\mathcal G}[v_\eps]+2v_\eps+v_\eps'(\pi-2a_-)-(v_\eps')^2\right).
\end{align*}
Since $v''(a)=0$ we have ${\mathcal G}[\tilde v_\eps](a)=0>\frac{a}{a_-}{\mathcal G}[v_\eps](a_-)$.  Also $\tilde v_\eps'(a)=v_\eps'(a_-)$, so that
\begin{align*}
\frac{\partial}{\partial t}\tilde v_\eps(a) &<{\mathcal G}[\tilde v_\eps]+2\tilde v_\eps
+\tilde v_\eps'(\pi-2a)-(\tilde v_\eps')^2-v_\eps'(1-\frac{a}{a_-})(\pi-\tilde v_\eps')\\
&<{\mathcal G}[\tilde v_\eps]+2\tilde v_\eps
+\tilde v_\eps'(\pi-2a)-(\tilde v_\eps')^2.
\end{align*}
The case $a\in(\pi-a_-,\pi)$ follows by symmetry.

Now for any $\eps>0$ we can apply Theorem \ref{thm:ODEcomparison} to show that $\Psi(\Omega_t,a)>\tilde f_\eps(a,t)$ (we leave it to the reader to check that the fact that $\tilde f_\eps$ is only $C^{1,1}$ and piecewise smooth is no obstacle).  Letting $\eps\to 0$ we deduce that $\Psi(\Omega_t,a)\geq f(a,t)=\Psi(\Theta_t,a)$ for all $a\in(0,\pi)$ and $t\in[0,T)$.
\end{proof}

\begin{corollary}\label{cor:curv.bound}
Under the conditions of Corollary \ref{cor:CSF.comparison}, the curvature $\curv$ of $\partial\Omega_t$
satisfies $\max_{\partial\Omega_t}\curv\leq \max_{\partial\Theta_t}\curv$.
\end{corollary}

\begin{proof}
This follows immediately from Corollary \ref{cor:CSF.comparison} and the asymptotic behaviour of the isoperimetric profile given in Proposition \ref{prop:asympt.isoperim.profile}.
\end{proof}

\section{Upper curvature bound from the Angenent solution}

In this section we compare with an explicit solution to produce an upper curvature bound for any embedded smooth solution of the normalized curve shortening flow equation \eqref{eq:NCSF}.  The `paperclip' solution of \eqref{eq:CSF} is given by
$$
\tilde\Theta_\tau=\left\{(\tilde x,\tilde y)\in\RR\times(-\pi/2,\pi/2):\ \E^
\tau\cosh(\tilde x)-\cos(\tilde y)\leq 0\right\},\quad \tau<0.
$$
This solution contracts to the origin with circular asymptotic shape as $\tau\to 0$.  In bounded regions it converges as $\tau\to-\infty$ to the parallel lines $y=\pm\frac{\pi}{2}$, while near the maxima of curvature it is asymptotic to the grim reaper $\{x=-\tau+\log 2+\log\cos y\}$.

\begin{figure}
\includegraphics[scale=0.6]{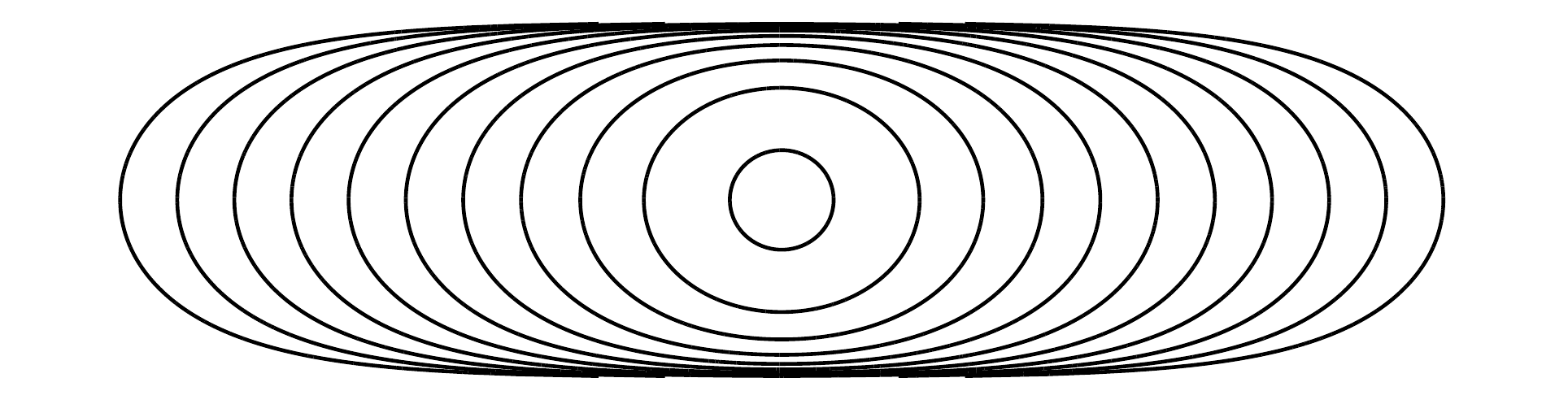}
\caption{The un-normalized paperclip for a range of $\tau<0$.}
\end{figure}

Corresponding to this is the solution of \eqref{eq:NCSF} given for $t\in\RR$ by
$$
\Theta_t=\left\{(x,y):\ |y|<\frac{{\pi}}{2}\E^t,\ \E^{-\frac12\E^{-2t}}\cosh\left(\E^{-t}x\right)-\cos\left(\E^{-t}y\right)\leq 0\right\}.
$$
The curvatures can be computed exactly:  Since $\tilde\Theta_\tau$ is a sub-level set of the convex function $G(x,y)=\E^\tau\cosh\tilde x-\cos\tilde y$, we have for $(\tilde x,\tilde y)\in\partial\tilde\Theta_\tau$
$$
\nor(\tilde x,\tilde y) = \frac{\nabla G}{|\nabla G|}=\frac{1}{\sqrt{\E^{2\tau}\sinh^2\tilde x+\sin^2\tilde y}}\begin{bmatrix}\E^{\tau}\sinh\tilde x\\\sin\tilde y\end{bmatrix}=\frac{1}{\sqrt{1-\E^{2\tau}}}\begin{bmatrix}\E^{\tau}\sinh\tilde x\\\sin\tilde y\end{bmatrix}
$$
so that
$$
\tang(\tilde x,\tilde y)=\begin{bmatrix}0&-1\\1&0\end{bmatrix}\nor(\tilde x,\tilde y) = 
\frac{1}{\sqrt{1-\E^{2\tau}}}\begin{bmatrix}-\sin\tilde y\\\E^\tau\sinh\tilde x\end{bmatrix}.
$$
The curvature is then given by
\begin{equation*}
\tilde\curv(\tilde x,\tilde y)=D_{\tang}\nor\cdot\tang=\frac{\E^\tau}{\sqrt{1-\E^{2\tau}}}\cosh\tilde x = \frac{1}{\sqrt{1-\E^{2\tau}}}\cos\tilde y.
\end{equation*}
The only critical points of $\tilde\curv$ are where $\tilde y=0$ or $\tilde x=0$, and the points of maximum
curvature lie on the $\tilde x$ axis and have value $(1-\E^{2\tau})^{-1/2}$.  The rescaled regions $\Theta_t$ therefore satisfy the conditions of Theorem \ref{thm:model.to.solution}, and have maximum curvature given by
$$
\curv_{\max} = \frac{\E^{-t}}{\sqrt{1-\E^{-\E^{-2t}}}} = 1+\frac14\E^{-2t}+O(\E^{-4t})\quad\text{as\ }t\to\infty.
$$

We claim that for any simply connected region $\Omega_0$ of area $\pi$ with smooth boundary $\gamma_0$, there exists $t_0$ such that $\Psi(\Omega_0,a)\geq\Psi(\Theta_{t_0},a)$ for all $a\in(0,\pi)$.  To see this, note that for fixed $a\in(0,\pi)$ we have $\Psi(\Theta_{t},a)=\pi\E^{t}(1+o(1))\to 0$ as $t\to-\infty$, since $\Theta_{t}$ is asymptotic to a pair of parallel lines with separation $\pi\E^{t}$.  The asymptotic grim reaper shape gives for $a>0$
$$
\Psi(\Theta_{t},a\E^{2t}) = \E^t\Psi(\mathfrak{G},a)(1+o(1))\quad\text{as\ }t\to-\infty,
$$
where $\mathfrak{G}$ is the grim reaper $\{x\leq\log\cos y,\ |y|<\pi/2\}$.  The existence of a suitable $t_0$ follows, and hence by Corollary \ref{cor:curv.bound} we have $\curv\leq \frac{\E^{-(t-t_0)}}{\sqrt{1-\E^{-\E^{-2(t-t_0)}}}}$, and so $\curv\leq 1+\frac14\E^{-2(t-t_0)}+O(\E^{-4t})$ as $t\to\infty$ for any closed curve evolving by the normalized curve shortening flow.

\section{Exterior isoperimetric profile and lower curvature bound}

In order to deduce long-time existence of the solution of normalized curve-shortening flow, it suffices to show that the curvature remains bounded.  The previous section gave an upper bound, and in this section we prove a lower bound by considering the exterior isoperimetric profile.  We begin with the analogue of Theorem \ref{thm:ODEcomparison} for the exterior profile:

\begin{theorem}\label{thm:extODEcomparison}
Let $f:\ \RR_+\times[0,\infty)\to\RR$ be continuous, smooth where both arguments are positive, concave in the first argument for each $t$, and such that $\limsup_{z\to 0}\frac{f(z,t)}{\sqrt{2\pi z}}<1$ and $\limsup_{z\to\infty}
\frac{f(z,t)}{\sqrt{4\pi z}}<1$, and
$$
\frac{\partial f}{\partial t}<-f^{-1}{\mathcal F}[ff',f^3f'']+f+f'(\pi-2a)-f(f')^2
$$
for all $a>0$ and $t\geq 0$.
Suppose $\gamma_t=\partial\Omega_t$ is a family of smooth embedded curves evolving by \eqref{eq:NCSF}  and satisfying $\Psi_\text{ext}(\Omega_0,a)>f(a,0)$ for all $a>0$, then $\Psi_{\text{ext}}(\Omega_t,a)>f(a,t)$ for all $t\geq 0$ and $a\in(0,\pi)$.
\end{theorem}

\begin{proof}
The proof is closely analogous to that of Theorem \ref{thm:ODEcomparison}.  We first establish conditions under which the isoperimetric exterior domains are simply connected and have a single boundary curve:

\begin{lemma}\label{lem:extconnecteddomains}
If $f:\ \RR_+\to\RR$ is strictly concave and strictly increasing, and $\Omega\subset\RR^2$ is a compact simply connected domain with $\Psi(\Omega,a)\geq f(a)$ for every $a\geq 0$, then every region $K\subset \RR^2\setminus\bar\Omega$ with $|\partial_{\RR^2\setminus\bar\Omega}K|=f(|K|)$ and $|K|>0$ is connected and simply connected.
\end{lemma}

\begin{proof}
As in the proof of Lemma \ref{lem:connecteddomains}, $K$ is connected since $f$ is strictly concave.  Now suppose that 
$\RR^2\setminus(\bar\Omega\cup\bar K)$ is not connected.  Then there exists a component $L$ of $\RR^2\setminus(\Omega\cup K)$ which is bounded.  Let $\tilde K$ be the interior of $\bar(K\cup L)$.  
Then every boundary component (relative to $\RR^2\setminus\bar\Omega$) of $\tilde K$ is a boundary component of $K$, so $|\partial_{\RR^2\setminus\bar\Omega}\tilde K|\leq|\partial_{\RR^2\setminus\bar\Omega}K|$, while $|\tilde K|>|K|$.  But then since $f$ is strictly increasing, we have
$$
|\partial_{\RR^2\setminus\bar\Omega}\tilde K|\leq |\partial_{\RR^2\setminus\bar\Omega}K| = f(|K|)<f(|\tilde K|)
$$
which contradicts the assumption of the Lemma.
Therefore $K$ and its complement in $\RR^2\setminus\bar\Omega$ are connected, so $K$ is simply connected.
\end{proof}

The behaviour of the exterior profile for small $a$ is determined by Proposition \ref{prop:asympt.isoperim.profile}.
We also need to establish the behaviour for large $a$:

\begin{lemma}\label{lem:ext-profile-asymp}
For $\Omega\subset\RR^2$ compact, $\lim_{a\to\infty}\frac{\Psi_{\text{ext}}(\Omega,a)}{\sqrt{4\pi a}}=1$.
\end{lemma}

\begin{proof}
The upper bound is trivial, since for any $a>0$ we can choose $K$ to be a ball of area $a$ which does not intersect $\Omega$, giving $f(a)\leq |\partial K|=\sqrt{4\pi a}$.  For the lower bound, let $K$ be an isoperimetric region of area $a$ in $\RR^2\setminus\bar\Omega$.  Then $\partial_{\RR^2}K\subset \partial_{\RR^2\setminus\bar\Omega}K\cup\partial\Omega$, so $|\partial_{\RR^2}K|\leq |\partial_{\RR^2\setminus\bar\Omega}K|+|\partial\Omega|$.
 By the isoperimetric inequality for the plane we have $|\partial_{\RR^2}K|\geq \sqrt{4\pi|K|}=\sqrt{4\pi a}$.  Combining these inequalities we find $f(|K|)\geq \sqrt{4\pi a}-|\partial\Omega|$.
\end{proof}

This guarantees that under the assumptions of Theorem \ref{thm:extODEcomparison}, at the first time where the inequality does not hold strictly, we must have equality for some $a\in(0,\infty)$.
The remainder of the proof is identical to that in Theorem \ref{thm:ODEcomparison} (except that since we are working with the exterior of $\Omega_t$, the normal direction and the curvature are replaced by their negatives throughout).
\end{proof}

To apply this we prove a result analogous to Theorem \ref{thm:model_isoperim}:

\begin{theorem}\label{thm:ext-model-isoperim}
 Let $\gamma=\partial\Omega$, where $\Omega$ is a smoothly bounded non-compact convex region with only one vertex and reflection symmetry in the $x$ axis.   Let $X:(-\pi/2,\pi/2)\to\RR^2$ be the map which takes $\theta$ to the point in $\gamma$ with outward normal direction $(\cos\theta,\sin\theta)$.  Then for each $\theta\in(0,\pi/2)$ there exists a unique constant curvature curve $\sigma_\theta$ which is contained in $\Omega$ and has endpoints at $X(\theta)$ and $X(-\theta)$ meeting $\gamma$ orthogonally.  Let $K^x_\theta$ denote the compact connected component of $\Omega\setminus\sigma_\theta$.   Then there exists a smooth, increasing diffeomorphism $\theta$ from $(0,\infty)$ to $(0,\pi/2)$ such that $K_a=K^x_{\theta(a)}$ has area $a$ for each $a\in(0,\infty)$, and the unique isoperimetric regions of area $a$ in $\Omega$ is $K_{a}$.
 \end{theorem}
 
 \begin{proof}
 By convexity, $\partial\Omega$ is defined by an embedding $X:\ (-\theta_0,\theta_0)\to\RR^2$ for some $\theta_0\in(0,\pi/2]$ which takes $\theta$ to the point in $\partial\Omega$ with outward normal direction $\theta$.   The argument of 
 \cite{Kuwert} shows that $\Psi(\Omega,a)^2$ is strictly concave, hence strictly increasing since it is defined and positive for all positive $a$.  By the argument in \cite{SZ} or Lemma 
 \ref{lem:extconnecteddomains} the boundary of any isoperimetric region is a single circular arc meeting $\partial\Omega$ orthogonally at both ends.  The argument of Theorem \ref{thm:model_isoperim} shows that there is only one candidate for an isoperimetric region for each $a>0$, which is that given in the Theorem.
 \end{proof}

To produce suitable solutions of the differential inequality we consider suitable non-compact solutions of the normalized flow:

\begin{theorem}\label{thm:ext-model-to-solution}
Let $\Omega_0$ be a non-compact convex subset of $\RR^2$, with smooth boundary curve $\gamma_0$ given by the image of a smooth embedding $X_0:\ S^1\to\RR^2$, and assume $\Omega_0$ is symmetric in the $x$ axis and has only one vertex.   
Let $X:\ S^1\times[0,T)\to\RR^2$ be the solution of \eqref{eq:NCSF} with initial data $X_0$.  Then for each $t\in[0,T)$, the region $\Omega_t$ enclosed by $\gamma_t=X(S^1,t)$ is a non-compact convex region symmetric in the $x$ axis, with only one vertex.   For each $t$, let $K_{a,t}$ be the family of isoperimetric regions for $\Omega_t$ constructed in Theorem \ref{thm:ext-model-isoperim}, and define $f(a,t)=\left|\partial_{\Omega_t}K_{a,t}\right|$.  Then $f:\ (0,\infty)\times[0,T)\to\RR$ is an increasing concave solution of the equation \eqref{eq:comparison.equation}
with $\lim_{a\to 0}\frac{f(a,t)}{\sqrt{2\pi a}}=1$, ${\mathcal F}[ff',f^3f'']>0$, and $\lim_{a\to\infty}\frac{f(a,t)}{\sqrt{4\pi a}}=1$.
\end{theorem}

The proof is the same as that of Theorem \ref{thm:model.to.solution}, using Theorem \ref{thm:ext-model-isoperim} instead of Theorem \ref{thm:model_isoperim}.
Arguing as in Corollary \ref{cor:CSF.comparison}, we deduce the following comparison theorem:

\begin{corollary}\label{cor:extcomparison}
Let $\{\Omega_t:\ 0\leq t<T\}$ be any smooth compact embedded solution of the normalized curve shortening flow \eqref{eq:NCSF}, and let $\{\Theta_t:\ 0\leq t<T\}$ be a solution of \eqref{eq:NCSF} for which $\Theta_0$ is a smoothly bounded non-compact convex region with reflection symmetry in the $x$ coordinate axes and exactly one vertex, such that $\Psi_\text{ext}(\Omega_0,a)\geq\Psi(\Theta_0,a)$ for every $a>0$.  Then
$\Psi_\text{ext}(\Omega_t,a)\geq\Psi(\Theta_t,a)$ for all $a>0$ and all $t\in[0,T)$.
\end{corollary}

The asymptotics for small $a$ of the exterior profile given in Proposition \ref{prop:asympt.isoperim.profile} then imply the following:

\begin{corollary}\label{cor:lower-curv-bound}
Under the conditions of Corollary \ref{cor:extcomparison}, $\min_{\partial\Omega_t}\curv\geq -\max_{\partial\Theta_t}\curv$.
\end{corollary}

Now we apply this for a particular choice of model region to deduce the required lower curvature bound:

\begin{theorem}\label{thm:lower-curv-bound}
For any compact embedded solution of \eqref{eq:NCSF} there exists $C$ such that 
$\curv(x,t)\geq -C\E^{-t}$ for $t>0$.
\end{theorem}

\begin{proof}
We choose as a comparison region a solution of \eqref{eq:NCSF} arising from a homothetically expanding solution of curve shortening flow (see \cite{EHentire}*{Theorem 5.1} or  \cite{Ishimura}) which we can construct as follows:  Define
$h:\ (-\theta_0,\theta_0)\to\RR$ implicitly by
$$
\theta = \int_{h(\theta)}^1\frac{dz}{\sqrt{1-z^2-C\log z}},
$$
where $\theta_0\in(0,\pi/2)$ is determined by $C>0$.  $\theta_0$ is strictly monotone in $C$ and approaches $0$ as $C\to \infty$ and approaches $\pi/2$ as $C\to 0$.  The curve given by the image
of the map $X$ in Equation \eqref{eq:X} on the interval $(-\theta_0,\theta_0)$ is then a complete convex curve asymptotic to the lines of angle $\pm\theta_0$ with a single critical point of curvature at $\theta=0$, at which point the curvature takes its maximum value of $1/C$.  At every point of the curve the equation $\curv=-C^{-1}\langle X,\nu\rangle$ holds.  Let $\Theta$ be the non-compact convex region enclosed by this curve.  Then the regions $\tilde\Theta_\tau=\sqrt{\frac{2\tau}{C}}\Theta$ satisfy the curve-shortening flow, and the rescaled regions $\Theta_t=r(t)\Theta$ satisfy the normalized curve-shortening flow \eqref{eq:NCSF}, where $r(t)=\sqrt{\frac{\E^{2t}-1}{C}}$ for $t>0$.

As $t=0$ the region $\Theta_t$ converges to the wedge of angle $2\theta_0$, so the isoperimetric profile is exactly $\sqrt{4\theta_0a}$ for $a>0$.  In particular for any smooth simply compact region $\Omega_0$, for sufficiently small $\theta_0$ we have
$\Psi_{\text{ext}}(\Omega_0,a)>\Psi(\Theta_0,a)$ for every $a$, and by continuity we also have
$\Psi_{\text{ext}}(\Omega_0,a)>\Psi(\Theta_\delta,a)$ for all $a$ for small $\delta>0$.  Corollary \ref{cor:extcomparison} gives $\curv\geq -1/(Cr(t))=-\frac{1}{\sqrt{C(\E^{2t}-1)}}$. 
\end{proof}

\begin{remark}
One could also apply the comparison theorem with $\Theta_t=\E^{t-t_0}{\mathfrak G}$ for sufficiently large $t_0$, where ${\mathfrak G}$ is the convex region enclosed by the grim reaper curve.  This gives the lower bound $\curv\geq -C\E^{-t}$ for some $C$.  The comparison used above is interesting because it implies curvature bounds for positive times, independent of any initial curvature bound, provided the initial exterior isoperimetric profile is bounded below by $C\sqrt{a}$ for some $C$, and the initial isoperimetric profile is bounded below by $C\min\{\sqrt{a},\sqrt{\pi-a}\}$.  
\end{remark}

\section{Proof of Grayson's theorem}

We have established upper and lower bounds on curvature for any compact simply connected region with boundary evolving by the normalized curve-shortening flow, with the upper curvature bound exponentially decaying to $1$ as $t\to\infty$.  The argument in \cite{AB1}*{Sections 3--4} applies, proving Grayson's theorem.


\end{document}